\documentclass[10pt, a4paper]{amsart}

\usepackage{mathtools}
\usepackage{hyperref}
\hypersetup{
    colorlinks=true,
    linkcolor=blue,
    filecolor=blue,      
    urlcolor=blue,
    citecolor=blue,
}
\usepackage{amsmath}
\usepackage{amsthm}
\usepackage{amssymb}
\usepackage[matha,mathx]{mathabx} 
\usepackage[svgnames]{xcolor}
\usepackage{tikz}
\usepackage{tikz-cd}
\usepackage{subcaption}
\usepackage{url}
\usepackage{stmaryrd}
\usepackage{bbm}
\usepackage{bm}
\usepackage{longtable}
\usepackage{tikz-qtree}
\usepackage{mathrsfs}
\usepackage{enumitem}

\usetikzlibrary{arrows, backgrounds}

\numberwithin{equation}{section}
\numberwithin{figure}{section}

\definecolor{mygray}{gray}{0.7}

\newtheorem{thmabc}{Theorem}

\newtheorem{corabc}[thmabc]{Corollary}
\newtheorem{thm}{Theorem}[section]
\newtheorem*{thm*}{Theorem}
\newtheorem*{con*}{Conjecture}
\newtheorem{lem}[thm]{Lemma}

\newtheorem{prop}[thm]{Proposition}
\newtheorem{cor}[thm]{Corollary}

\newtheorem{conj}[thm]{Conjecture}

\theoremstyle{definition}
\newtheorem{defn}[thm]{Definition}
\newtheorem{remark}[thm]{Remark}

\newtheorem{quest}[thm]{Question}

\newtheorem*{acknowledgements}{Acknowledgements}

\DeclareMathOperator{\SL}{SL}
\DeclareMathOperator{\UIV}{UIV}
\DeclareMathOperator{\Spec}{Spec}
\DeclareMathOperator{\binv}{binv}
\DeclareMathOperator{\Ig}{I}
\DeclareMathOperator{\GSp}{GSp}
\DeclareMathOperator{\host}{HoSt}

\DeclareMathOperator{\HS}{HS}

\DeclareMathOperator{\Id}{Id}
\DeclareMathOperator{\inv}{inv}
\DeclareMathOperator{\maj}{maj}
\DeclareMathOperator{\Mat}{Mat}

\DeclareMathOperator{\des}{des}

\DeclareMathOperator{\GL}{GL}

\DeclareMathOperator{\diag}{diag}

\newcommand{\sta}{\phi}
\newcommand{\poly}[1]{\mathscr{P}^{#1}}
\newcommand{\tin}{\mathsf{n}}
\newcommand{\tud}{\textup{d}}
\newcommand{\HLS}{\mathsf{HLS}}
\newcommand{\consymA}[2]{c_{#1,#2}^{\mathsf{A}}}
\newcommand{\consymC}[2]{c_{#1,#2}^{\mathsf{C}}}
\newcommand{\dual}[1]{\widetilde{#1}}

\newcommand{\mcEpr}{\mathcal{E}^{\textup{pr}}}
\newcommand{\Dif}{\mathrm{inc}}
\newcommand{\Hecke}{\mathsf{H}}

\newcommand{\bfx}{\bm{x}}

\newcommand{\N}{\mathbb{N}}

\newcommand{\Q}{\mathbb{Q}}
\newcommand{\R}{\mathbb{R}}

\newcommand{\Z}{\mathbb{Z}}
\newcommand{\C}{\mathbb{C}}
\renewcommand{\phi}{\varphi}

\newcommand{\heckeoperator}[4]{T^{\mathsf{#4}}_{#1,#2,#3}}
\newcommand{\heckeoperatoralt}[3]{T^{\mathsf{#3}}_{#1,#2}}
\newcommand{\hopa}[3]{\heckeoperator{#1}{#2}{#3}{A}}
\newcommand{\hopc}[3]{\heckeoperator{#1}{#2}{#3}{C}}

\newcommand{\hopalta}[2]{\heckeoperatoralt{#1}{#2}{A}}

\newcommand{\birk}[3]{B_{#1,#2,#3}(q)}
\renewcommand{\leq}{\leqslant}
\renewcommand{\geq}{\geqslant}

\renewcommand{\epsilon}{\varepsilon}

\newcommand{\Des}{\mathrm{Des}}

\newcommand{\lri}{\mathfrak{o}}

\newcommand{\define}[1]{\textbf{\textit{#1}}}

\def \mcC {\mathcal{C}}

\def \mcE {\mathcal{E}}
\def \mcH {\mathcal{H}}
\def \mcL {\mathcal{L}}

\def \mcP {\mathcal{P}}

\def \mcZ {\mathcal{Z}}
\def \Gri {\mathcal{O}}

\def \Z {\mathbb{Z}}
\def \Zp {\mathbb{Z}_p}
\def \Q {\mathbb{Q}}
\def \Qp {\mathbb{Q}_p}
\def \N {\mathbb{N}}
\def \mfp {\mathfrak{p}}

\newcommand{\Hermite}[1]{\delta(#1)}
\newcommand{\HSsp}{\overline{\HS}}
\newcommand{\transpose}{\mathrm{t}}
\newcommand{\msfX}{\mathsf{X}}
\newcommand{\msfA}{\mathsf{A}}
\newcommand{\msfC}{\mathsf{C}}
\newcommand{\lattice}[1]{\Lambda_{#1}}
\newcommand{\lZA}[3]{\mcZ_{#1, #2, #3}^{\msfA}}
\newcommand{\lZC}[3]{\mcZ_{#1, #2, #3}^{\msfC}}
\newcommand{\lZX}[3]{\mcZ_{#1, #2, #3}^{\msfX}}
\newcommand{\gZA}[2]{\mcZ_{#1, #2}^{\msfA}}
\newcommand{\gZC}[2]{\mcZ_{#1, #2}^{\msfC}}
\newcommand{\gZX}[2]{\mcZ_{#1, #2}^{\msfX}}
\newcommand{\EHCA}[2]{\mcC_{#1, #2}^{\msfA}}
\newcommand{\EHCC}[2]{\mcC_{#1, #2}^{\msfC}}
\newcommand{\EHCX}[2]{\mcC_{#1, #2}^{\msfX}}

\allowdisplaybreaks

\title[Symplectic Hecke eigenbases from Ehrhart polynomials]{Symplectic Hecke eigenbases from \\ Ehrhart polynomials}

\author{Claudia Alfes}
\author{Joshua Maglione}
\author{Christopher Voll}

\address{Fakult\"at f\"ur Mathematik, Universit\"at Bielefeld, D-33501
  Bielefeld, Germany}

\email{alfes@math.uni-bielefeld.de, C.Voll.98@cantab.net}

\address{School of Mathematical and Statistical Sciences, University of Galway, Ireland.}

\email{joshua.maglione@universityofgalway.ie}

\keywords{%
  Ehrhart polynomials,
  Hecke eigenbases,
  Hecke eigenfunctions,
  Hecke operators, 
  lattice polytopes, 
  local functional equations, 
  $p$-adic integration,
  polytope valuations,
  quadratic permutation statistics,
  spherical functions, 
  symplectic lattices, 
  symplectic similitudes, 
  zeta functions%
}

\subjclass[2020]{20C08, 52B20, 43A90, 05A15, 11M41}

\begin{document}

\begin{abstract}
  For $n\in\N$ and $\ell\in\{0,1,\dots,n\}$, we consider the function extracting
  the $\ell$th coefficient of the Ehrhart polynomials of lattice polytopes
  in~$\R^n$. These functions form a basis of the space of unimodular invariant
  valuations. We show that, in even dimensions, these functions are in fact
  simultaneous symplectic Hecke eigenfunctions. We leverage this and apply the
  theory of spherical functions and their associated zeta functions to prove
  analytic, asymptotic, and combinatorial results about the arithmetic functions
  averaging $\ell$th Ehrhart coefficients.
\end{abstract}

\date{\today} \maketitle

\thispagestyle{empty}

\section{Introduction and main results}\label{sec:intro}

Ehrhart polynomials are fundamental invariants of lattice polytopes;
see, for instance,~\cite{BV/97}
and~\cite{JS/18}. Let~$\mathscr{E}_{n,\ell} =
\mathscr{E}^{\Lambda_0}_{n,\ell}$ be the function extracting the
$\ell$th coefficient of these polynomials of a lattice polytope with
respect to a fixed lattice~$\Lambda_0$ in~$\R^n$.  These functions
form a basis of the real $(n+1)$-dimensional space $\UIV_n$ of
unimodular invariant valuations on lattice polytopes with respect
to~$\Lambda_0$; see \cite{BK/85}. In the present paper we focus on the
variation of the functions $\mathscr{E}^{\Lambda}_{n,\ell}$ with
refinements $\Lambda$ of the lattice~$\Lambda_0$ and connections with
the theory of Hecke operators.

Exhibiting bases consisting of simultaneous Hecke eigenfunctions is a
recurrent theme in the theory of automorphic forms. For a prime $p$,
the $p$-adic symplectic Hecke ring is a polynomial ring generated by
Hecke operators $\hopc{n}{k}{p}$ for $k \in [n]_0 =
\{0,1,\dots,n\}$. In contrast to the classical setup in the theory of
modular forms, we consider Hecke operators acting via averaging on
functions on polytopes, such as the Ehrhart
coefficients~$\mathscr{E}_{n,\ell}$;
see~\eqref{eqn:HeckeC-action}. Theorem~\ref{thmabc:ehr.sat} shows
that, in even dimensions, these functions are simultaneous Hecke
eigenfunctions---yielding a symplectic Hecke eigenbasis of the
space~$\UIV_{2n}$. We also explicate the parameters which, by the
theory of spherical functions, determine the associated Hecke
eigenvalues via the Satake isomorphism~$\Omega_{n,p}^{\mathsf{C}}$;
see~\eqref{def:sat.iso}.  These ``spherical Ehrhart parameters''
\eqref{equ:ehr.par} are defined in terms of the $\Q$-valued ring
homomorphisms~$ \psi_{n,\ell,p}^{\mathsf{C}}$ specified
in~\eqref{def:psi}.

\begin{thmabc}\label{thmabc:ehr.sat} 
  Let $k\in[n]_0$ and $\ell\in[2n]_0$. There exists a polynomial
  $\Phi_{n,k,\ell}^{\mathsf{C}}\in\Z[Y]$ such that, for all
  primes~$p$,
  \begin{align*}
    \hopc{n}{k}{p}\mathscr{E}_{2n,\ell} &=
    \Phi_{n,k,\ell}^{\mathsf{C}}(p)\mathscr{E}_{2n,\ell}, &
    \Phi_{n,k,\ell}^{\mathsf{C}}(p) &=
    \psi_{n,\ell,p}^{\mathsf{C}}(\Omega_{n,p}^{\mathsf{C}}(\hopc{n}{k}{p})).
  \end{align*}
  Additionally, these polynomials satisfy
  \begin{align*}
    \dfrac{\Phi_{n,k,2n-\ell}^{\mathsf{C}}(Y)}{\Phi_{n,k,\ell}^{\mathsf{C}}(Y)} &= \begin{cases}
      Y^{n-\ell} & \text{ if } k = 0, \\
      Y^{2(n - \ell)} & \text{ if } k \in [n]. 
    \end{cases}
  \end{align*}
\end{thmabc}

The polynomials $\Phi_{n,k,\ell}^{\mathsf{C}}$ are given explicitly in
Definition~\ref{def:Phi}. Theorem~\ref{thm:A} is an analogue of
Theorem~\ref{thmabc:ehr.sat} in type~$\msfA$ due to Gunnells and
Rodriguez Villegas. In~\cite{GRV/07} they initiated the study of
Ehrhart coefficients via Hecke operators. To exhibit our subsequent
results simultaneously in types $\msfA$ and $\msfC$, we write $\msfX
\in \{\msfA,\msfC\}$ and denote by $\tin$ the ambient dimensions $n$
or $2n$, respectively.

Let $\mathscr{P}^{\Lambda_0}$ be the set of lattice polytopes in
$\R^\tin$ with respect to~$\Lambda_0 = \Z^{\tin}$, which we assume to
be a symplectic lattice in type~$\msfC$ (see
Section~\ref{sec:lattices}). For a lattice $\Lambda$ with $\Lambda_0
\subseteq \Lambda \subset \Q^{\tin}$ and $\ell\in [\tin]_0$, we denote
by
\begin{align}\label{def:E}
  \mathscr{E}_{\tin,\ell}^{\Lambda} : \poly{\Lambda_0} 
  \longrightarrow \Q, \quad P \longmapsto
  c_\ell\left(E_{P}^{\Lambda}\right)
\end{align}
the function isolating the $\ell$th coefficient of the Ehrhart
polynomial $E_{P}^{\Lambda}$ of $P$ with respect to~$\Lambda$. For a
polytope $P$ with $\mathscr{E}_{\tin, \ell}(P) \neq 0$, we define the
function
\begin{equation}\label{def:Cpl}
  \EHCX{P}{\ell} : \N \longrightarrow \Q, \quad m \longmapsto \dfrac{1}{\mathscr{E}_{\tin, \ell}(P)} \sum_{\Lambda} \mathscr{E}_{\tin, \ell}^{\Lambda}(P),
\end{equation} 
where the sum ranges over all (symplectic, in type $\msfC$) lattices $\Lambda$
with $|\Lambda : \Lambda_0| = m$. The arithmetic function $\EHCX{P}{\ell}$ thus
records the average value of the $\ell$th Ehrhart coefficients of $P$ over all
(symplectic, in type $\msfC$) refinements of $\Lambda_0$ with fixed co-index.
Our next result shows that these functions are multiplicative.

\begin{thmabc}\label{thmabc:mult}
  Let $\ell\in[\tin]_0$ and $P$ be as above. For coprime integers
  $a,b\in\N$,
  \begin{align*}
    \EHCX{P}{\ell}(ab) &= \EHCX{P}{\ell}(a)\EHCX{P}{\ell}(b) . 
  \end{align*}
\end{thmabc}

We also obtain asymptotic statements on the accumulative growth
of~$\EHCX{P}{\ell}$. For functions $f,g:\R \to \R$, write $f(N)\sim
g(N)$ if $\displaystyle\lim_{N\to \infty}f(N)/g(N) = 1$. For
$i,j\in\N_0$, let $\delta_{i,j}$ be the Kronecker delta.

\begin{thmabc}\label{thmabc:asy}
  There exist real numbers $c_{n,\ell}^{\mathsf{X}}\in\R_{>0}$ such that as $N\to \infty$, 
  \begin{align*}
    \sum_{m=1}^N \EHCX{P}{\ell}(m) ~&\sim~ \begin{cases}
      \consymA{n}{\ell} N^{\max(n, \ell+1)} (\log N)^{\delta_{n-1,\ell}} & \text{ if } \mathsf{X}=\mathsf{A}, \\[0.25em]
      \consymC{n}{\ell} N^{\frac{n+1}{2} + \frac{1+\max(0, \ell-n)}{n}} (\log N)^{\delta_{n,\ell}} & \text{ if } \mathsf{X}=\mathsf{C}. 
    \end{cases}
  \end{align*}
\end{thmabc}

Proposition~\ref{prop:type-A-asymptotics} describes the numbers
$c_{n,\ell}^{\mathsf{A}}$ precisely in terms of special values of the
Riemann zeta function. The situation in type $\mathsf{C}$ is subtler
still; see Section~\ref{subsec:nonneg}.

\subsection{Ehrhart--Hecke zeta functions}
\label{sec:intro-zeta}

To prove Theorems~\ref{thmabc:mult} and~\ref{thmabc:asy}, we leverage
Theorem~\ref{thmabc:ehr.sat} by using the theory of spherical
functions and their local zeta functions.

Let $\lri$ be a compact discrete valuation ring and $K$ its field of
fractions. Let $\mfp\subset \lri$ be the maximal ideal and $\pi$ a
uniformizing element. Suppose the cardinality of $\lri/\mfp$ is
$q$. Let $|\cdot |_{\mfp}$ be the $\mfp$-adic norm. For a prime $p$,
we let $\Z_p$ be the ring of $p$-adic integers and $\Q_p$ for its
field of fractions. Sometimes we restrict to rings of the form $\lri =
\Zp$ so that results can be interpreted within Ehrhart theory.

Let $\mathrm{GSp}_{2n}$ be the algebraic group scheme of symplectic
similitudes.  We set $\mathrm{G}_n^+(K) = \mathrm{GSp}_{2n}(K)\cap
\Mat_{2n}(\lri)$. A function $\omega : \GSp_{2n}(K)\to \C$ is
\define{spherical} if $\omega$ is
$\mathrm{GSp}_{2n}(\lri)$-bi-invariant, a Hecke eigenfunction with
respect to the convolution product, and $\omega(1)=1$. It is well
known from work of Satake~\cite{Satake/63} that spherical functions
are parametrized by elements of $\C^{n+1}$;
see~\cite[Ch.~1]{Macdonald/71}.  Theorem~\ref{thmabc:ehr.sat} provides
such an element---what we call the $\ell$th spherical Ehrhart
parameters---defined in~\eqref{equ:ehr.par}. For $\ell\in\Z$, let
\begin{equation}\label{def:spher}
  \omega_{n,\ell, \lri}: \GSp_{2n}(K) \longrightarrow \Q
\end{equation}
be the corresponding spherical function.  Theorem~\ref{thmabc:ehr.sat}
implies that, for $\ell\in[2n]_0$ and $\lri = \Zp$, this spherical
function may be interpreted in terms of a normalized average of the
$\ell$th coefficient of the Ehrhart polynomial; see
Section~\ref{subsec:cor.spher} for details.

Following \cite[Sec.~V.5]{Macdonald/95}, we attach a zeta function to
each of the spherical functions $\omega_{n,\ell,\lri}$.  To formulate
these zeta functions as $\mfp$-adic integrals, we let $\mu$ be the
Haar measure on $\mathrm{GSp}_{2n}(K)$ normalized so that
$\mu(\mathrm{GSp}_{2n}(\lri)) = 1$. Let $\varphi$ be the
characteristic function of~$\mathrm{G}_n^+(K)$.

The \define{local Ehrhart--Hecke zeta function} of type $\mathsf{C}$
is
 \begin{align}\label{def:Z} 
  \lZC{n}{\ell}{\lri}(s) &= \int_{\GSp_{2n}(K)} \varphi(g)
  \left|\det(g)\right|_{\mfp}^s \omega_{n,\ell,\lri}(g^{-1})
  \,\mathrm{d}\mu.
  \end{align}
We note that this is a slight variation of the zeta function in
Macdonald~\cite[Sec.~V.5]{Macdonald/95}: there, the $n$th root of the
determinant is used in the integrand.

By Theorem~\ref{thmabc:ehr.sat}, local Ehrhart--Hecke zeta functions
are instantiations of the Hecke series associated with $\GSp_{2n}(K)$;
see \cite[Sec.~V.5]{Macdonald/95}. By \cite[Cor.~1.11]{MV/24} these
zeta functions thus satisfy local functional equations:

\begin{corabc}\label{cor:func-eq}
  For all $\ell\in\Z$,
  \begin{align}\label{equ:funeq}
    \left. \lZC{n}{\ell}{\lri}(s)\right|_{q
      \rightarrow q^{-1}} &= (-1)^{n+1} q^{n^2+\ell-2ns} 
    \lZC{n}{\ell}{\lri}(s).
  \end{align}
\end{corabc}

\begin{remark}
  Local functional equations as the one established in
  Corollary~\ref{cor:func-eq} are a well-studied phenomenon. Work of
  Igusa \cite{Igusa/89} implies \eqref{equ:funeq} in the case
  $\ell=0$. That a functional equation of the form \eqref{equ:funeq}
  holds for all spherical functions $\omega$ with non-zero parameters
  is a consequence of \cite[Cor.~1.11]{MV/24}, which establishes a
  functional equation for the symplectic Hecke series; see
  Section~\ref{subsec:proof.thmabc:BIgu}. Du Sautoy and Lubotzky
  extended Igusa's work and applied it to zeta functions of nilpotent
  groups in \cite{duSL/96}, an approach further developed in
  \cite{B/11,BGS/22}. Denef and Meuser \cite{DM/91} proved functional
  equations of the form~\eqref{equ:funeq} for more general $p$-adic
  integrals. Their ideas in turn were developed and adjusted to the
  context of numerous local zeta functions associated with groups,
  rings, and modules; see, for instance, \cite{Voll/10, AKOV1/13,
    Rossmann/18, MV1, CSV/24}. ``Self-reciprocities'' akin to the
  functional equation~\eqref{equ:funeq} arising from enumerative
  combinatorics feature prominently in~\cite{BS/18}.
\end{remark}

Theorems~\ref{thmabc:ehr.sat} and~\ref{thmabc:mult} and
Proposition~\ref{prop:spherical} imply that, for $\lri = \Zp$, the
local Ehrhart--Hecke zeta functions are Euler factors of the
associated \define{global Ehrhart--Hecke zeta functions}
\begin{align}\label{equ:global.Z}
  \gZC{n}{\ell}(s) &= \sum_{m=1}^{\infty} \EHCC{P}{\ell}(m) m^{-s}, 
\end{align}
the Dirichlet generating series for the arithmetic function
$\EHCC{P}{\ell}$ defined in~\eqref{def:Cpl}.

\begin{cor}\label{cor:global.Z}
  For all $\ell\in [2n]_0$, 
  \begin{align*}
    \gZC{n}{\ell}(s) = \prod_{p \textup{ prime}} \lZC{n}{\ell}{p}(s).
  \end{align*}
\end{cor}

Lastly, we leverage the connection with Hecke series to obtain
explicit formulae for the local Ehrhart--Hecke zeta
functions~$\lZC{n}{\ell}{\lri}(s)$.  To formulate
them we introduce further notation. Given subsets $I, J\subseteq
\N_0$, we define $P_{I,J} = (I \times \{0\}) \cup (J \times \{1\})$, a
subposet of $\N_0^2$, totally ordered by the lexicographical
ordering. For $x,y\in P_{I,J}$, we say that $y$ covers $x$ if $x<y$
and $x\leq z < y$ implies $x=z$. Set
\begin{align*}
  C_{I, J} &= \{(i,j)\in I\times J ~|~ (j,1) \text{ covers } (i,0) \}. 
\end{align*}
We write $I-1 = \{i-1 ~|~ i\in I\}$ and define
\begin{align}\label{def:Psi}
  \Psi_{n, I, J}(Y) &= \binom{n-1}{(I - 1) \cup (J - 1)}_{Y}
  ~\prod_{(i,j)\in C_{I, J-1}} (1 - Y^{j - i + 1}) \in \Z[Y].
\end{align}
Here, the first factor on the right denotes the $Y$-multinomial
coefficient; see~\eqref{def:bino}.

\begin{thmabc}\label{thmabc:BIgu}
  For all $\ell\in \Z$ and with $t = q^{-s}$,
  \begin{align*}
  \lZC{n}{\ell}{\lri}(s) &= \sum_{I,J\subseteq [n]} \Psi_{n,I,J}(q^{-1}) \prod_{i\in I}
    \dfrac{q^{\binom{i+1}{2} + i(n-i)}t^n}{1-q^{\binom{i+1}{2} +
        i(n-i)}t^n} \prod_{j\in J} \dfrac{q^{\binom{j+1}{2} + j(n-j) - n +
        \ell}t^n}{1-q^{\binom{j+1}{2} + j(n-j) - n + \ell}t^n} .
  \end{align*}
\end{thmabc}

In the case $n=2$, Corollary~\ref{cor:global.Z} and
Theorem~\ref{thmabc:BIgu} yield a pleasingly simple formula:
\begin{align}\label{eqn:Z2ell}
  \gZC{2}{\ell}(s) &=
  \dfrac{\zeta(2s-2)\zeta(2s-3)\zeta(2s-\ell)\zeta(2s-\ell-1)}{\zeta(4s-\ell-2)},
\end{align}
where $\zeta(s)$ is the Riemann zeta function. For larger $n$, the
local zeta functions are much more complicated. In particular,
$\gZC{n}{\ell}(s)$ is not, in general, a quotient of translates of the
Riemann zeta function. More examples are given in
Section~\ref{subsec:EH.exa}.

From Theorem~\ref{thmabc:BIgu} we observe that the value $\ell=n$ is
somewhat special. Indeed Theorems~\ref{thmabc:ehr.sat}
and~\ref{thmabc:asy} and Conjecture~\ref{conj:l=n} supply further
evidence for the noteworthiness of this special $\ell$-value.

Our next result establishes a curious duality relation among the
functions $\lZC{n}{\ell}{\lri}$.

\begin{corabc}\label{corabc:refl}
  For all $\ell\in \Z$,
  \begin{align*}
    \lZC{n}{2n-\ell}{\lri}(s) &= \lZC{n}{\ell}{\lri}\left(s - \frac{n-\ell}{n}\right) . 
  \end{align*}
\end{corabc}

\subsection{Notation}\label{subsec:not}
Throughout $n\in \N = \{1, 2, \dots \}$ and $\lri$ is a compact
discrete valuation ring (cDVR) with maximal ideal $\mfp$ and residue
field cardinality $q$. By $\pi\in \mfp \setminus \mfp^2$ we denote an
arbitrary but fixed uniformizer of $\lri$.

Setting $J=\left(
\begin{smallmatrix} 0 & I_n \\ -I_n & 0 \end{smallmatrix} \right)\in\Mat_{2n}(\Z)$ allows us to describe, for a ring $R$, the group of $R$-rational points of the group scheme $\GSp_{2n}$ as
\begin{align*}
  \mathrm{GSp}_{2n}(R) &= \left\{ A \in \GL_{2n}(R) \mid
  A^{\transpose} J A = \kappa(A)J,\ \text{for some } \kappa(A)\in
  R^\times \right\} .
\end{align*}
We set 
\begin{align}\label{eqn:Gamma-G+}
  \Gamma_n &= \mathrm{GSp}_{2n}(\lri), & \mathrm{G}_n^+(K) &= \mathrm{GSp}_{2n}(K)\cap
  \Mat_{2n}(\lri) .
\end{align}

For $a,b\in \N_0=\N\cup\{0\}$ with $0\leq b\leq a$, the $Y$-binomial
coefficient is
\begin{equation*}
  \binom{a}{b}_Y = \prod_{i=1}^{b} \frac{1-Y^{a-i+1}}{1-Y^i} \in \Z[Y].
\end{equation*}
For $S = \{s_1,\dots, s_k\}\subseteq [n]$ with $s_1<\cdots < s_k$, the
$Y$-multinomial coefficient is
\begin{align}\label{def:bino}
  \binom{n}{S}_Y = \binom{n}{s_k}_Y\binom{s_k}{s_{k-1}}_Y \cdots \binom{s_2}{s_1}_Y \in \Z[Y].
\end{align}

We summarize notation used throughout the paper in the following table.

\begingroup \renewcommand\arraystretch{1.18}
\begin{longtable}{l|l|l}
  Symbol & Description & Reference \\ \hline \hline
  $\poly{\Lambda_0}$ & set of lattice polytopes & Section~\ref{sec:intro}\\ 
  $\mathscr{E}_{n,\ell}^{\Lambda}$, $\mathscr{E}_{n,\ell} = \mathscr{E}_{n,\ell}^{\Lambda_0}$
  &$\ell$th Ehrhart coefficients & Equation~\eqref{def:E}
  \\
  $\EHCX{P}{\ell}$ & average $\ell$th Ehrhart coefficients & Equation~\eqref{def:Cpl} \\
  $\mathrm{G}_n^+(K)$, $\Gamma_n$ & integral symplectic similitudes & Equation~\eqref{eqn:Gamma-G+} \\
  $\omega_{n,\ell,\lri}$ & spherical functions & Equation~\eqref{def:spher} \\
  $\lZX{n}{\ell}{\lri}$, $\gZX{n}{\ell}$, $\lZX{n}{\ell}{F}$
  &Ehrhart--Hecke zeta functions & Equations~\eqref{def:Z} \&~\eqref{def:EP.Z}\\
  $\Psi_{n,I,J}$ & polynomials in $\lZC{n}{\ell}{\lri}$ & Equation~\eqref{def:Psi} \\ $\mathscr{D}_{n,k,\lri}^{\mathsf{C}}$
  &symplectic lattices&
  Equation~\eqref{def:Dnk}\\ $D_{n,k,\lri}^{\mathsf{C}}$ & symplectic diagonal
  matrices &
  Equation~\eqref{def:diag}\\ $\mathcal{H}_{n,\lri}^{\mathsf{C}}$ &
  symplectic Hecke ring&
  Section~\ref{sec:Hecke-Satake}\\ $\hopc{n}{k}{\lri}$ & symplectic
  Hecke operators &
  Section~\ref{sec:Hecke-Satake}\\ 
  $\psi_{n,\ell,\lri}^{\mathsf{C}}$ &ring homomorphisms &
  Equation~\eqref{def:psi} \\ $\Omega_{n,\lri}^{\mathsf{C}}$ & Satake isomorphism
  &Definition~\ref{def:Omega_k}\\
  $\Phi_{n,k,\ell}^{\mathsf{C}}$ & Hecke eigenvalues & Definition~\ref{def:Phi} 
\end{longtable}
\endgroup

\section{Symplectic Hecke rings and polytopes}\label{sec:pre.sym}

We recall definitions regarding the symplectic Hecke rings and define
their action on unimodular invariant valuations.

\subsection{Lattices and Lagrangian subspaces}\label{sec:lattices}

Let $\mcP_n$ denote the set of integer partitions of at most $n$
parts---these are weakly decreasing sequences of non-negative integers
with at most $n$ positive integers. For a partition $\lambda$, let
$\lambda'$ be the conjugate partition---viewed by taking the transpose
of the associated Young diagram---and let $|\lambda|$ be the sum of
its parts. We define the finite $\lri$-module
\begin{align*}
  C_{\lambda}(\lri) &= \bigoplus_{i\in\N} \lri/\mfp^{\lambda_i}. 
\end{align*}
The (\define{elementary divisor}) \define{type} of a finite-index
lattice $\Lambda\leq \lri^n$, written $\lambda(\Lambda)$, is the
partition $\lambda\in \mcP_n$ such that~$\lri^n/\Lambda \cong
C_\lambda(\lri)$. The following lemma goes back at least to work of
Birkhoff~\cite{Birkhoff/35}; see \cite[Sec.~1.4]{Butler/94} for an
historical account.

\begin{lem}[{\cite[Lem.~1.4.1]{Butler/94}}]\label{lem:Birkhoff-formula}
  Let $\lambda,\mu\in\mathcal{P}_n$ with $\lambda_i\geq \mu_i$ for all $i\in[n]$. The number of submodules of $C_{\lambda}(\lri)$ of type $\mu$ is given by 
  \begin{equation*}
    \prod_{a\in \N} q^{\mu_a'(\lambda_a'-\mu_a')} \binom{\lambda_a' - \mu_{a+1}'}{\lambda_a' - \mu_a'}_{q^{-1}}.
  \end{equation*}
\end{lem}

We fix an $\lri$-basis $\{e_1,\dots, e_n\}$ of $\lri^n$. For
$i\in [n]_0$, let $V_{(i)} = \langle e_1,\dots, e_i\rangle \leq \lri^n$, so
\[
  0 = V_{(0)} < V_{(1)} < \cdots < V_{(n-1)} < V_{(n)} = \lri^n,
\] 
where each $V_{(i+1)}/V_{(i)}$ is torsion free. Assume that $\varpi_i : \lri^n \to
V_{(i)}$ is given by $e_j\mapsto e_j$ for $j\in [i]$ and $e_k\mapsto 0$ for
$k\in \{i+1,i+2, \dots,n\}$. 

Let $\lambda_{(i)}(\Lambda)\in \mathcal{P}_i$ be
the type of $\varpi_i(\Lambda)$ in $V_{(i)}$ for a finite-index lattice $\Lambda\leq \lri^n$.
We recall, e.g.\ from \cite[Sec.~3.2]{MV/24}, the notion of the
\define{Hermite composition}
\begin{equation}\label{def:herm}
  \delta(\Lambda) = (\delta_1(\Lambda), \dots,
  \delta_n(\Lambda))\in\N_0^n
\end{equation} 
of a lattice $\Lambda$ where $\delta_i(\Lambda) =
|\lambda_{(i)}(\Lambda)| - |\lambda_{(i-1)}(\Lambda)|$ for $i\in [n]$.
In the case $\lri = \Zp$, these numbers are the exponents of the
$p$-powers adorning the diagonal of an integral matrix in Hermite
normal form representing~$\Lambda$.

We consider $K^{2n}$ as a non-degenerate symplectic space with
symplectic form~$\langle\,,\rangle$.  A \define{symplectic
  $\lri$-lattice} $\Lambda$ is a lattice of $K^{2n}$ generated by the
rows of some $g\in \mathrm{G}_n^+(K)$. Such a lattice is
\define{primitive} if the induced form
$\overline{\langle\,,\rangle}:\Lambda/\mfp\Lambda \times
\Lambda/\mfp\Lambda\to \lri/\mfp$ is non-degenerate. Homothety classes
of symplectic $\lri$-lattices are in bijection with the special
vertices of the affine building of type $\widetilde{{\mathsf{C}}}_n$
and thickness $q$ associated with~$\GSp_{2n}(K)$;
see~\cite[Sec.~20.1]{Garrett}.

Let $V$ be a finite-dimensional symplectic vector space, that is, $V$
is equipped with a nondegenerate alternating bilinear form
$\langle\,,\rangle$. For a subspace $U\leq V$, define
\[ 
  U^{\perp} = \langle v\in V ~|~ \forall u\in U,\ \langle u, v\rangle = 0\rangle. 
\]
A subspace $U\leq V$ is \define{totally isotropic} if $U\leq
U^{\perp}$. A \define{Lagrangian subspace} $L$ of $V$ is
totally-isotropic of maximal dimension $\frac{1}{2}\dim
V$. We denote by $\mathrm{LGr}(V)$ the \define{Lagrangian
  Grassmannian} of $V$, i.e.\ the set of Lagrangian subspaces of~$V$.

\begin{lem}\label{lem:ti-fun}
  Let $V \cong\mathbb{F}_p^{2n}$ be a symplectic vector space. For a
  $1$-dimensional subspace~$U\leq V$ we have $\# \{L \in
  \mathrm{LGr}(V) \mid U\leq L \} = \prod_{i=1}^{n-1} (1 + q^{i})$.
\end{lem}

\begin{proof}
  Observe that $U^{\perp}$ has dimension $2n-1$ and $U\leq
  U^{\perp}$. Thus the Lagrangian subspaces of $V$ containing $U$ are
  in bijection with the Lagrangian subspaces of $U^{\perp}/U \cong
  \mathbb{F}_p^{2n-2}$.  By \cite[Lem.~(5.1)]{Krieg/90} we have
  $\#\mathrm{LGr}(\mathbb{F}_p^{2n-2}) = \prod_{i=1}^{n-1} (1 +
  q^{i})$. \qedhere
 \end{proof}

\subsection{The symplectic Hecke ring and Satake isomorphism}\label{sec:Hecke-Satake}

Let $\mcH_{n,\lri}^{\mathsf{C}}$ be the Hecke ring associated with the
pair $(\mathrm{G}_n^+(K), \Gamma_n)$; see~\eqref{eqn:Gamma-G+}. Thus
$\mcH_{n,\lri}^{\mathsf{C}}\otimes_{\Z}\C$ is the Hecke algebra of
continuous functions $\mathrm{G}_n^+(K)\to \C$ with compact support
that are bi-invariant with respect to $\Gamma_n$. A $\Z$-basis for
$\mcH_{n,\lri}^{\mathsf{C}}$ is given by the characteristic functions
of the $\Gamma_n$-double cosets; see \cite[Sec.~V.5]{Macdonald/95}. We
write $\mcH_{n,p}^{\mathsf{C}}$ in the case $\lri = \Zp$ and use
similar abbreviations throughout.

We define a set of $n+1$ algebraically independent generators
of~$\mcH_{n,\lri}^{\mathsf{C}}$. Let $\diag(a_1,\dots, a_r)$ be the $r\times r$
diagonal matrix with diagonal entries $a_1,\dots, a_r$. Let
\begin{equation}\label{def:diag} 
  D_{n,k,\lri}^{\mathsf{C}} = \begin{cases} \diag(1^{(n)}, \pi^{(n)}) & \textup{ if }
    k=0,\\ \diag(1^{(n-k)}, \pi^{(k)}, (\pi^2)^{(n-k)}, \pi^{(k)}) &\textup{
      if }k\in[n],
  \end{cases}
\end{equation}
where $\alpha^{(\beta)} = (\alpha,\dots,\alpha)$ with $\beta$
repetitions.  For $k\in [n]_0$, let $\hopc{n}{k}{\lri}\in
\mcH_{n,\lri}^{\mathsf{C}}$ be the characteristic function
of~$\Gamma_nD_{n,k,\lri}^{\mathsf{C}}\Gamma_n$. By
\cite[Thm.~(6.2)]{Krieg/90}, we have
\begin{equation}\label{eqn:Hecke-poly-ring}
  \mcH_{n,\lri}^{\mathsf{C}} = \Z[\hopc{n}{0}{\lri},\, \dots,\, \hopc{n}{n}{\lri}]. 
\end{equation}
We further define
\begin{equation}\label{def:Dnk}
  \mathscr{D}_{n,k,\lri}^{\mathsf{C}} = \Gamma_n\setminus \Gamma_n D_{n,k,\lri}^{\mathsf{C}}\Gamma_n = \left\{\Gamma_n g\in \Gamma_n\setminus \mathrm{G}_n^+(K) ~\middle|~  \Gamma_ng\Gamma_n = \Gamma_nD_{n,k,\lri}^{\mathsf{C}}\Gamma_n\right\}. 
\end{equation}
The set $\mathscr{D}_{n,k,\lri}^{\mathsf{C}}$ is in bijection with the
set of symplectic lattices with prescribed symplectic elementary
divisors given by $D_{n,k,\lri}^{\mathsf{C}}$.

Let $\bm{x} = (x_0,\dots, x_n)$ be variables. We consider the image of
the symplectic Hecke ring $\mcH_{n,\lri}^{\mathsf{C}}$ under a
normalization (see \cite[Sec.~8]{Gross:Satake}) of the Satake
isomorphism
\begin{align}\label{def:sat.iso}
  \Omega_{n,\lri}^{\mathsf{C}} : \mcH^{\mathsf{C}}_{n,\lri}\otimes_{\Z} \Q \longrightarrow \Q[\bm{x}]^W,
\end{align}
where $W$ is the finite Weyl group of type $\mathsf{C}_n$. To make
$\Omega_{n,\lri}^{\mathsf{C}}$ explicit on the generators $\{\hopc{n}{k}{\lri}
~|~ k\in [n]_0\}$, we establish further notation. We denote by
$\mathrm{Sym}_{n,r}(F)$ the set of symmetric $n\times n$ matrices of rank $r$
over a field~$F$. For $a,b\in [n]_0$ with $a+b\leq
n$, we define
\begin{equation}\label{def:lat.a.b}
  \mathcal{L}_{a,b}(\lri^n) = \left\{ \Lambda \leq \lri^n ~\middle|~
  \lambda(\Lambda) = \left(2^{(b)}, 1^{(a)}, 0^{(n-a-b)}\right)
  \right\} .
\end{equation}
We enumerate these sets in Lemma~\ref{lem:birkhoff}. We further define
\begin{align}\label{eqn:omega_ab}
  \omega_{a,b}(\lri^n) &= \sum_{\Lambda\in\mathcal{L}_{a,b}(\lri^n)}~
  \prod_{i=1}^n (x_iq^{-i})^{\delta_i(\Lambda)} \in \Q[x_1,\dots,x_n],
\end{align}
where $\delta_i(\Lambda)$ is the $i$th component of the Hermite
composition of $\Lambda$ defined in~\eqref{def:herm}.

\begin{defn}\label{def:Omega_k}
  For $k\in [n]$ we set
  \begin{align*} 
    \Omega_{n,\lri}^{\mathsf{C}}(\hopc{n}{0}{\lri}) &= x_0\prod_{i=1}^n(1 + x_i), \\
    \Omega_{n,\lri}^{\mathsf{C}}(\hopc{n}{k}{\lri}) &= x_0^2 \sum_{a=k}^n \sum_{b=0}^{n-a} q^{b(a+b+1)} \omega_{a,b}(\lri^n) (\# \mathrm{Sym}_{a,a-k}(\mathbb{F}_q)) .
  \end{align*}
\end{defn}

By~\eqref{eqn:Hecke-poly-ring}, this determines the map
$\Omega_{n,\lri}^{\mathsf{C}}$ from~\eqref{def:sat.iso}. By
\cite[Thm.~3.3.30~(1)]{Andrianov/87}, $\Omega_{n,\lri}^{\mathsf{C}}$
is an isomorphism. For $\ell\in\Z$ we further define a ring
homomorphism
\begin{align}\label{def:psi}
    \psi_{n,\ell,\lri}^{\mathsf{C}} : \Q[\bm{x}] \longrightarrow \Q & \quad\quad  
    x_i \longmapsto 
    \begin{cases} 
      q^{\ell} & \text{ if }i=0,\\
      q^i & \text{ if }i\in [n-1], \\
      q^{n-\ell} & \text{ if } i=n.
    \end{cases}
\end{align}
We call the parameters
\begin{equation}\label{equ:ehr.par}
  (q^{\ell}, q^1, q^2, \dots, q^{n-1}, q^{n-\ell}) \in \Q^{n+1}
\end{equation}
the \define{$\ell$th spherical Ehrhart parameters};
cf.~\cite[V.3]{Macdonald/95}.

\begin{quest}\label{que:other.sat.par}
  Do the $\ell$th spherical Ehrhart parameters allow for meaningful
  Ar\-chi\-me\-dean analogues? What other parameters of the form $(p^{e_0},
  p^{e_1}, \dots, p^{e_n})$ for $e_i\in\Z$ yield solutions to
  interesting lattice enumeration problems?
\end{quest}

\subsection{A Hecke action on polytope invariants}
For the remainder of this section, we restrict to $\lri = \Z_p$. We
define
$\mathrm{G}_n^{+}(\Q)=\GSp_{2n}(\Q)\cap\mathrm{Mat}_{2n}(\Z)$. For
$M\in \mathrm{G}_n^+(\Q)$, we let $M\cdot P$ be the polytope obtained
from $P$ by the natural (left) action of $M$ on the vertices of
$P$. Hence, $\mathrm{G}_n^+(\Q)$ acts on $\mathscr{P}^{\Lambda_0}$,
and we extend this action to $\mathrm{G}_n^+(\Q_p)$ using the
following elementary lemma, where $\iota :
\mathrm{Mat}_{2n}(\Z)\hookrightarrow \mathrm{Mat}_{2n}(\Z_p)$ is the
(dense) inclusion. We omit its simple proof.

\begin{lem}\label{lem:non-Arch-to-Arch}
  Let $g\in \mathrm{G}_n^{+}(\Q_p)$. There exists
  $M_g\in\mathrm{G}_n^{+}(\Q)$, unique up to $\GSp_{2n}(\Z)$, such
  that $\iota(M_g) \in\Gamma_ng$ and $\det(M_g) = |\det(g)|_{p}^{-1}$.
\end{lem}

We define an action of $\mcH_{n,p}^{\mathsf{C}}$ on the space
$\UIV_{2n}$ of unimodular invariant valuations.  For $g\in
\mathrm{G}_n^+(\Q_p)$, let $M_g\in \mathrm{G}_n^+(\Q)$ be as in
Lemma~\ref{lem:non-Arch-to-Arch}. For $P\in\mathscr{P}^{\Lambda_0}$,
we set $ g\cdot P = M_g \cdot P$.  Let $T\in \mcH_{n,p}^{\mathsf{C}}$
be the characteristic function of the double coset $\Gamma_ng\Gamma_n$
for some $g\in\mathrm{G}_n^+(\Q_p)$. For $f \in \mathrm{UIV}_{2n}$, we
define $Tf:\mathscr{P}^{\Lambda_0} \to \R$ via
\begin{align}\label{eqn:HeckeC-action}
  (Tf)(P) &= \sum_{\Gamma_nh\in \Gamma_n\setminus \Gamma_ng\Gamma_n} f(h\cdot P).
\end{align}
We extend \eqref{eqn:HeckeC-action} additively to all of
$\mcH_{n,p}^{\mathsf{C}}$.

\section{Ehrhart coefficients as Hecke eigenfunctions: proof of Theorem~\ref{thmabc:ehr.sat}}

In Section~\ref{subsec:poly.eigen} we establish polynomial expressions
for the Hecke eigenvalues of the functions~$\mathscr{E}_{2n,\ell}$. In
Section~\ref{sec:Ehr-Sat} we show that the $\ell$th spherical Ehrhart
parameters \eqref{equ:ehr.par} yield these eigenvalues. We prove
Theorem~\ref{thmabc:ehr.sat} in Section~\ref{sec:proof.thmabc.ehr.sat}
and interpret the associated spherical function $\omega_{n,\ell,p}$ in
terms of Ehrhart coefficients in Section~\ref{subsec:cor.spher}.

\subsection{Polynomial expressions for the Hecke eigenvalues}
\label{subsec:poly.eigen}

We assume $\lri = \Z_p$ throughout this section; some statements hold,
\emph{mutatis mutandis}, for general cDVR. We start by enumerating
these sets $\mathscr{D}_{n,k,p}^{\mathsf{C}}$ defined
in~\eqref{def:Dnk}.

\begin{defn}\label{def:Delta}
  For $k\in[n]_0$, set
  \begin{align*}
    \Delta_{n,k}(Y) &= \prod_{i=1}^{n-k}(1+Y^{k+i}) \cdot \begin{cases}
      1 & \text{ if } k=0,\\
      Y^{\binom{n-k+1}{2}} \binom{n}{k}_{Y} & \text{ if }k>0.
    \end{cases}
  \end{align*}
\end{defn}

\begin{prop}[Krieg~\cite{Krieg/90}]\label{prop:krieg} 
  For $k\in[n]_0$, we have
  \[
    \# \mathscr{D}_{n,k,p}^{\mathsf{C}} = \Delta_{n,k}(p).
  \]
\end{prop}

\begin{proof}
  For $k=0$, apply \cite[Lem.~(5.1)]{Krieg/90}; for $k>0$, apply
  \cite[Cor.~(7.3)]{Krieg/90}.
\end{proof}

For $g\in \mathrm{G}_n^+(\Q_p)$, we denote by $\lattice{g}$ the
lattice generated by the rows of the inverse transpose of $g$ so that
$\Lambda_0\leq \lattice{g}$. We need to enumerate the lattices in
$\{\lattice{g} ~|~ \Gamma_ng\in\mathscr{D}_{n,k,p}^{\mathsf{C}} \}$
that contain a prescribed vector. For $i\in \{1,2\}$ and $k\in[n-1]$,
let $x_i\in p^{-i}\Lambda_0 \setminus p^{1-i} \Lambda_0$, and set
\begin{align*}
  \xi_{n, k, p}(x_i) &= \# \left\{\Gamma_ng\in\mathscr{D}_{n,k,p}^{\mathsf{C}} ~\middle|~ x_i \in \lattice{g}\right\}.
\end{align*}
The following lemma determines this quantity and shows that it depends
only on~$i$ and not on the specific vector~$x_i$.

\begin{lem}\label{lem:pts-in-lats}
  For $k\in [n-1]$,
  \begin{align*} 
    \xi_{n, k, p}(x_1) &= \dfrac{p^{n+k} - 1}{p^{2n} - 1} \Delta_{n, k}(p), & \xi_{n, k,p}(x_2) &= \dfrac{p^{2n} - p^{n+k}}{p^{4n} - p^{2n}} \Delta_{n, k}(p).
  \end{align*}
\end{lem}

\begin{proof}
  We enumerate the following sets in two different ways: for $i\in \{1,2\}$
  define 
  \begin{align*}
    S_i &= \left\{(H, \lattice{g}) ~\middle|~ \Gamma_ng\in\mathscr{D}_{n,k,p}^{\mathsf{C}}, \ \Lambda_0 \leq H \leq \lattice{g}, \
    H/\Lambda_0 \cong C_{p^i}\right\} ,
  \end{align*}
  where $C_{p^i}$ is the cyclic $p$-group of order $p^i$. On the one
  hand, if $\Gamma_ng\in\mathscr{D}_{n,k,p}^{\mathsf{C}}$, then
  \begin{align}\label{eqn:S_i-1}
    \# S_i = \# \mathscr{D}_{n,k,p}^{\mathsf{C}} \cdot \# \{H \leqslant \lattice{g} ~|~ \Lambda_0\leqslant H,\ H/\Lambda_0\cong C_{p^i} \}.
  \end{align}
  On the other hand, if $H\leqslant p^{-2}\Lambda_0$ with $\Lambda_0\leqslant H$
  and $H/\Lambda_0\cong C_{p^i}$, then
  \begin{align}\label{eqn:S_i-2}
    \# S_i = \# \{H'\leqslant C_{p^2}^{2n} ~|~ H'\cong C_{p^i} \} \cdot \#\{\Gamma_ng\in\mathscr{D}_{n,k,p}^{\mathsf{C}} ~|~ H\leqslant \lattice{g} \}.
  \end{align} 
  For such an $x_i\in p^{-i}\Lambda_0 \setminus p^{1-i} \Lambda_0$ and
  $\Gamma_ng\in\mathscr{D}_{n,k,p}^{\mathsf{C}}$, we have $x_i\in\lattice{g}$
  if and only if $\langle x_i\rangle + \Lambda_0 \leq \lattice{g}$. Thus 
  \begin{align*}
    \xi_{n,k,p}(x_i) &= \#\{\Gamma_ng\in\mathscr{D}_{n,k,p}^{\mathsf{C}} ~|~ \langle x_i\rangle + \Lambda_0\leqslant \lattice{g} \}.
  \end{align*}
  By~\eqref{eqn:S_i-1} and~\eqref{eqn:S_i-2} and Proposition~\ref{prop:krieg}, we have 
  \begin{align*}
    \xi_{n,k,p}(x_i) &= \dfrac{\# \{H \leqslant C_p^{2k}\oplus C_{p^2}^{n-k} ~|~ H\cong C_{p^i} \}}{\# \{H'\leqslant C_{p^2}^{2n} ~|~ H'\cong C_{p^i} \}} \Delta_{n,k}(p).
  \end{align*}
  We apply Lemma~\ref{lem:Birkhoff-formula} to conclude the proof:
  \begin{align*}
    \# \{H \leqslant C_p^{2k}\oplus C_{p^2}^{n-k} ~|~ H\cong C_{p^i} \} &= \begin{cases}
      p^{2n-2}\binom{n-k}{1}_{p^{-1}} & \text{ if } i = 2, \\
      p^{n+k-1}\binom{n+k}{1}_{p^{-1}} & \text{ if } i =1,
    \end{cases}\\
    \# \{H'\leqslant C_{p^2}^{2n} ~|~ H'\cong C_{p^i} \} &= p^{(2n-1)i} \binom{2n}{1}_{p^{-1}}. \qedhere
  \end{align*}
\end{proof}

We explicate the polynomials paraphrased in
Theorem~\ref{thmabc:ehr.sat}.

\begin{defn}\label{def:Phi}
  For $k\in[n]$ and $\ell\in\Z$, set 
  \begin{align*}
    \Phi_{n, 0, \ell}^{\mathsf{C}}(Y) &= \dfrac{Y^{\ell} + Y^n}{1 + Y^n}\Delta_{n,0}(Y) = (Y^{\ell} + Y^n) \prod_{i=1}^{n-1}(1 + Y^i),\\
    \Phi_{n,k,\ell}^{\mathsf{C}}(Y) &= \dfrac{\Delta_{n,k}(Y)}{Y^{2n}-1} \left(Y^{2\ell} - Y^{2\ell - n + k} + Y^{\ell+n+k} - 2Y^{\ell} + Y^{\ell-n+k} + Y^{2n} - Y^{n+k}\right).
  \end{align*}
\end{defn}

We note that the $\Phi_{n,k,\ell}^{\mathsf{C}}(Y)$ are generally
Laurent polynomials. The polynomiality of these functions may not be
obvious from Definition~\ref{def:Phi}. We record the polynomials
$\Phi^{\mathsf{C}}_{n,k,\ell}(Y)$ for $n\in\{1,2\}$ in
Table~~\ref{tab:2}.

\begin{table}[h]
  \centering
  \begin{tabular}{r|cc}
    $n=1$& $k=0$ & $k=1$\\ \hline
    $\ell = 0$ & $Y + 1$ & $1$ \rule{0pt}{2.2ex}\\
    $\ell = 1$ & $2 Y$ & $Y$\\
    $\ell = 2$ & $Y^{2} + Y$ & $Y^{2}$\\
  \end{tabular}

\vspace{1em}

  \begin{tabular}{r|ccc}
    $n=2$& $k=0$ & $k=1$ & $k=2$\\ \hline
    $\ell = 0$ & $Y^{3} + Y^{2} + Y + 1$ & $Y^{4} + Y^{3} + Y^{2} + Y$ & $1$ \rule{0pt}{2.2ex}\\
    $\ell = 1$ & $Y^{3} + 2 Y^{2} + Y$ & $2 Y^{4} + Y^{3} + 2 Y^{2} - Y$ & $Y$\\
    $\ell = 2$ & $2 Y^{3} + 2 Y^{2}$ & $Y^{5} + 3 Y^{4} + Y^{3} - Y^{2}$ & $Y^{2}$\\
    $\ell = 3$ & $Y^{4} + 2 Y^{3} + Y^{2}$ & $2 Y^{6} + Y^{5} + 2 Y^{4} - Y^{3}$ & $Y^{3}$\\
    $\ell = 4$ & $Y^{5} + Y^{4} + Y^{3} + Y^{2}$ & $Y^{8} + Y^{7} + Y^{6} + Y^{5}$ & $Y^{4}$\\
  \end{tabular} 
  \caption{The polynomials $\Phi^{\mathsf{C}}_{n,k,\ell}(Y)$ for $n\in\{1,2\}$ and $\ell\in[2n]_0$.}
  \label{tab:2}
\end{table}

We are now in a position to prove the first part of
Theorem~\ref{thmabc:ehr.sat}.

\begin{prop}\label{prop:part-1}
  Let $k\in[n]_0$ and $\ell\in[2n]_0$. For all primes $p$,
  \begin{align*}
    \hopc{n}{k}{p}\mathscr{E}_{2n,\ell} &= \Phi_{n,k,\ell}^{\mathsf{C}}(p)\mathscr{E}_{2n,\ell}.
  \end{align*}
\end{prop}

\begin{proof}
  We consider three different cases: $k=n$, $k\in [n-1]$, and $k=0$.

  First consider $k=n$. Thus $\#\mathscr{D}_{n,n}^{\mathsf{C}} =
  \Delta_{n,n}(p) = 1$, enumerating the unique
  $\Gamma_n$-coset of $p\Id_{2n}$. Thus,
  \begin{align*}
    \hopc{n}{n}{p}\mathscr{E}_{2n,\ell}(P) =
    \mathscr{E}_{2n,\ell}(pP) = p^{\ell}\mathscr{E}_{2n,\ell}(P)
    = \Phi_{n,n,\ell}^{\mathsf{C}}(p)\mathscr{E}_{2n,\ell}(P).
  \end{align*}

  We assume now that $k\in [n-1]$, so we have 
  \begin{align*}
    \bigcup_{\Gamma_ng\in\mathscr{D}_{n,k,p}^{\mathsf{C}}} \lattice{g} = p^{-2}\Lambda_0
  \end{align*}
  as sets. We stratify $p^{-2}\Lambda_0\cap P$ into disjoint sets: for
  $a\in \N$, define
  \begin{align*}
    L_0(a) &= \Lambda_0\cap aP, & L_1(a) &= (p^{-1}\Lambda_0\setminus \Lambda_0)\cap aP, & L_2(a) &= (p^{-2}\Lambda_0\setminus p^{-1}\Lambda_0)\cap aP.
  \end{align*}
  Their respective cardinalities are $E_P(a)$, $E_P(pa) - E_P(a)$, and
  $E_P(p^2a) - E_P(pa)$, where $E_P$ is the Ehrhart polynomial of $P$ with
  respect to $\Lambda_0$. We determine which $\lattice{g}$ contain $L_i(a)$ for
  $\Gamma_ng\in\mathscr{D}_{n,k,p}^{\mathsf{C}}$. We have $L_0(a)\subseteq
  \lattice{g}$ for all $\Gamma_ng\in\mathscr{D}_{n,k,p}^{\mathsf{C}}$. For $i\in
  \{1,2\}$, let $x_i\in L_i(1)$. Thus by Proposition~\ref{prop:krieg} and
  Lemma~\ref{lem:pts-in-lats},
  \begin{align*}
    \hopc{n}{k}{p}\mathscr{E}_{2n,\ell}(P) &= (\#\mathscr{D}_{n,k,p}^{\mathsf{C}}) \mathscr{E}_{2n,\ell}(P) + \xi_{n,k,p}(x_1) \left(\mathscr{E}_{2n,\ell}(pP) - \mathscr{E}_{2n,\ell}(P)\right) \\
    &\quad + \xi_{n,k,p}(x_2) \left(\mathscr{E}_{2n,\ell}(p^2P) - \mathscr{E}_{2n,\ell}(p P)\right) \\ 
    &= \Delta_{n,k}(p) \!\left(1 + \dfrac{p^{n+k} - 1}{p^{2n} - 1}\left(p^{\ell} - 1\right) + \dfrac{p^{2n} - p^{n+k}}{p^{4n} - p^{2n}} \left(p^{2\ell} - p^{\ell}\right)\right) \! \mathscr{E}_{2n,\ell}(P) \\
    &= \Phi_{n,k,\ell}^{\mathsf{C}}(p) \mathscr{E}_{2n,\ell}(P).
  \end{align*}

  The final case is $k=0$, the structure of the argument is similar. Here, 
  \begin{align*}
    \bigcup_{\Gamma_ng\in\mathscr{D}_{n,0,p}^{\mathsf{C}}} \lattice{g} = p^{-1}\Lambda_0.
  \end{align*}
  The map $\mathscr{D}_{n,0,p}^{\mathsf{C}}
  \rightarrow\mathrm{LGr}(p^{-1}\Lambda_0/\Lambda_0), \quad
  \Gamma_ng\mapsto \lattice{g}/\Lambda_0$ is a bijection. For
  $a\in\N$, each $x\in L_1(a)$ determines a line $U = \langle x\rangle
  + \Lambda_0$ in the symplectic space
  $V=p^{-1}\Lambda_0/\Lambda_0$. By Lemma~\ref{lem:ti-fun}, we have
  $\# \{L \in \mathrm{LGr}(V) \mid U\leq L \} = \Delta_{n-1,0}(p)$,
  whence
  \begin{align*}
    \hopc{n}{0}{p}\mathscr{E}_{2n,\ell}(P) &=
    (\#\mathscr{D}_{n,0,p}^{\mathsf{C}})\mathscr{E}_{2n,\ell}(P) +
    \Delta_{n-1,0}(p)
    \left(\mathscr{E}_{2n,\ell}(p P) -
    \mathscr{E}_{2n,\ell}(P)\right) \\ 
    &=
    \Delta_{n-1,0}(p) \left(p^{\ell} + p^n\right)
    \mathscr{E}_{2n,\ell}(P) = \Phi_{n,0,\ell}^{\mathsf{C}}(p)
    \mathscr{E}_{2n,\ell}(P). \qedhere
  \end{align*}
\end{proof}

\subsection{Parameters for the Hecke eigenvalues} 
\label{sec:Ehr-Sat}

Let $\lri$ be an arbitrary cDVR. The following theorem yields the
second set of equations in Theorem~\ref{thmabc:ehr.sat}. Recall
from~\eqref{def:sat.iso} the normalized Satake
isomorphism~$\Omega_{n,\lri}^{\mathsf{C}}$, from~\eqref{def:psi} the
maps~$\psi_{n,\ell,\lri}^{\mathsf{C}}$, from Definition~\ref{def:Phi}
the polynomials~$\Phi_{n,k,\ell}^{\mathsf{C}}$ and from
\eqref{eqn:Hecke-poly-ring} the Hecke operators~$\hopc{n}{k}{\lri}$.

\begin{thm}\label{thm:Ehr-Sat}
  For $k\in[n]_0$ and $\ell\in\Z$,
  \begin{align*} 
    \Phi_{n,k,\ell}^{\mathsf{C}}(q) &=
    \psi_{n,\ell,\lri}^{\mathsf{C}}(\Omega_{n,\lri}^{\mathsf{C}}(\hopc{n}{k}{\lri})).
  \end{align*} 
\end{thm} 

Our proof of Theorem~\ref{thm:Ehr-Sat} uses the explicit polynomial expressions
for the Hecke operators $\hopc{n}{k}{\lri}$ under the Satake
isomorphism~$\Omega_{n,\lri}^{\mathsf{C}}$; see Definition~\ref{def:Omega_k}.

As the operator $\hopc{n}{0}{\lri}$ differs significantly from the
operators $\hopc{n}{k}{\lri}$ for $k\in[n]$, we can prove
Theorem~\ref{thm:Ehr-Sat} easily in the case $k=0$:

\begin{prop}\label{prop:k_0}
  For $\ell\in\Z$,
  \begin{align*}
    \Phi_{n,0,\ell}^{\mathsf{C}}(q) &= \psi_{n,\ell, \lri}^{\mathsf{C}}(\Omega_{n,\lri}^{\mathsf{C}}(\hopc{n}{0}{\lri})). 
  \end{align*} 
\end{prop}

\begin{proof} 
  By Definitions~\ref{def:Phi} and~\ref{def:Omega_k},
  \begin{align*} 
    \Phi_{n,0,\ell}^{\mathsf{C}}(q) &=
    (q^n+q^{\ell})\prod_{i=1}^{n-1}(1+q^i) = q^{\ell}(1 +
    q^{n-\ell})\prod_{i=1}^{n-1}(1+q^i) =
    \psi_{n,\ell,\lri}^{\mathsf{C}}(\hopc{n}{0}{\lri}). \qedhere
  \end{align*}
\end{proof}

Our strategy to prove Theorem~\ref{thm:Ehr-Sat} for positive $k$ has
two parts.  First we show, in Corollary~\ref{cor:ell-0}, that the
result holds in the base case~$\ell=0$.  Second we establish, in
Proposition~\ref{prop:eqn.diff}, the identity of differences
\begin{equation}\label{eqn:id-diff}
  \psi_{n,\ell,\lri}^{\mathsf{C}}(\Omega_{n,\lri}^{\mathsf{C}}(\hopc{n}{k}{\lri})) -
  \psi_{n,\ell-1,\lri}^{\mathsf{C}}(\Omega_{n,\lri}^{\mathsf{C}}(\hopc{n}{k}{\lri})) = \Phi_{n,k,\ell}^{\mathsf{C}}(q)
  - \Phi_{n,k,\ell-1}^{\mathsf{C}}(q).
\end{equation}  

\subsubsection{The base case~$\ell=0$} 

Recall the polynomial $\Delta_{n,k}$ from
Definition~\ref{def:Delta}. The main theorem we prove in this
subsection is the following.

\begin{thm}\label{thm:ell=0}
  For $k\in [n]$,
  \begin{equation*}
    \Delta_{n,k}(q) = \sum_{a=k}^n \sum_{b=0}^{n-a} q^{b(a+b+1)} \psi_{n,0,\lri}^{\mathsf{C}}(\omega_{a,b}(\lri^n))(\# \mathrm{Sym}_{a, a-k}(\mathbb{F}_q)).
  \end{equation*} 
\end{thm}

We apply Theorem~\ref{thm:ell=0} to prove Theorem~\ref{thm:Ehr-Sat}
for $\ell = 0$:

\begin{cor}\label{cor:ell-0}
  For $n\in\N$ and $k\in [n]$,
  \[
    \Phi_{n,k,0}^{\mathsf{C}}(q) = \Delta_{n,k}(q) = \psi_{n,0,\lri}^{\mathsf{C}}(\Omega_{n,\lri}^{\mathsf{C}}(\hopc{n}{k}{\lri})).
  \]
\end{cor}

\begin{proof}
  By Definition~\ref{def:Omega_k} and~\eqref{def:psi},
  \begin{align}\label{eqn:Satake-transform}
    \psi_{n,0,\lri}^{\mathsf{C}}(\Omega_{n,\lri}^{\mathsf{C}}(\hopc{n}{k}{\lri})) &= \sum_{a=k}^n \sum_{b=0}^{n-a} q^{b(a+b+1)} \psi_{n,0,\lri}^{\mathsf{C}}(\omega_{a,b}(\lri^n))(\# \mathrm{Sym}_{a,a-k}(\mathbb{F}_q)) .
  \end{align}
  As $\Phi_{n,k,0}^{\mathsf{C}} = \Delta_{n,k}$, the corollary follows
  by Theorem~\ref{thm:ell=0} and \eqref{eqn:Satake-transform}.\qedhere
\end{proof}

Recall the finite set of lattices $\mathcal{L}_{a,b}(\lri^n)$ in $\lri^n$;
see~\eqref{def:lat.a.b}. We set 
\begin{align*}
  \birk{n}{a}{b} &= \# \mathcal{L}_{a,b}(\lri^n).
\end{align*}
These cardinalities are given by an application of
Lemma~\ref{lem:Birkhoff-formula}.

\begin{lem}\label{lem:birkhoff}
  For $a,b\in[n]_0$ with $a+b\leq n$,
  \[ 
    \birk{n}{a}{b}= q^{(a+b)(n-a-b)+b(n-b)} \binom{n}{\{b,a+b\}}_{q^{-1}}.
  \]
\end{lem}

\begin{lem}\label{lem:no-S}
  For $b,k\in [n]_0$,
  \begin{align}\label{eqn:Birkhoff-sum}
    \sum_{a=k}^n q^{b(a+b+1)}\birk{n}{a}{b}(\# \mathrm{Sym}_{a,a-k}(\mathbb{F}_q)) &= q^{\binom{n-b-k+1}{2} + b(b+k+1)} \birk{n}{k}{b}.
  \end{align} 
\end{lem}

\begin{proof} 
  For $a\in \{k,k+1,\dots, n\}$, it follows from
  Lemma~\ref{lem:birkhoff} that 
  \begin{align*}
    q^{b(a+b+1)}\birk{n}{a}{b} &= q^{b(2n-b+1)}\binom{n}{b}_{q^{-1}}
    q^{a(n-a-b)}\binom{n-b}{a}_{q^{-1}}.
  \end{align*}
  Thus,~\eqref{eqn:Birkhoff-sum} is equivalent to 
  \begin{align}\label{eqn:middle-SMM}
    \sum_{a=k}^{n-b} (\# \mathrm{Sym}_{a, a-k}(\mathbb{F}_q)) q^{a(n-a-b)} \binom{n-b}{a}_{q^{-1}} &= q^{k(n-k-b) + \binom{n-b-k+1}{2}} \binom{n-b}{k}_{q^{-1}} .
  \end{align}
  From~\cite[Eq.~(3.4)]{StasinskiVoll/14}, we have 
  \begin{equation}\label{eqn:card-sym}
    \begin{split}
      \# \mathrm{Sym}_{a, a-k}(\mathbb{F}_q) &= q^{\binom{a+1}{2} - \binom{k+1}{2}} \prod_{c=1}^{\lfloor(a-k)/2\rfloor} \left(1 - q^{-2c}\right)^{-1} \prod_{d=1}^{a-k} \left(1 - q^{-d-k}\right) \\
      &= q^{\binom{a+1}{2} - \binom{k+1}{2}} \binom{a}{k}_{q^{-1}} \prod_{d=1}^{\lceil (a-k)/2 \rceil}\left(1 - q^{-2d+1}\right).
    \end{split}
  \end{equation} 
  Using \eqref{eqn:card-sym}, we further simplify~\eqref{eqn:middle-SMM} to
  obtain the following equivalent equation:
  \begin{align}\label{eqn:middle-SMM-2}
    \sum_{a=k}^{n-b} q^{\binom{a+1}{2} + a(n-a-b)} \binom{n-b-k}{a-k}_{q^{-1}} \prod_{d=1}^{\lceil (a-k)/2 \rceil}(1 - q^{-2d+1}) &= q^{\binom{n-b+1}{2}} .
  \end{align}
  We re-index the sum in \eqref{eqn:middle-SMM-2}, replace $n-b-k$ with $m$, and
  simplify the $q$-powers. Thus,~\eqref{eqn:Birkhoff-sum} is equivalent to
  \begin{align}\label{eqn:another-sym-equation}
    \sum_{a=0}^{m} q^{\binom{m+1}{2}-\binom{m+1-a}{2}}\binom{m}{a}_{q^{-1}} \prod_{d=1}^{\lceil a/2 \rceil}(1 - q^{-2d+1}) &= q^{\binom{m+1}{2}} .
  \end{align}
  But using \eqref{eqn:card-sym} again, we see that the summands on the left side
  of~\eqref{eqn:another-sym-equation} are $\#\mathrm{Sym}_{m,a}(\mathbb{F}_q)$.
  Clearly the number of symmetric $m\times m$ matrices (of all possible ranks)
  over $\mathbb{F}_q$ is the right side of~\eqref{eqn:another-sym-equation}.
  Hence, \eqref{eqn:another-sym-equation} holds, implying the lemma.
\end{proof}

In order to prove Theorem~\ref{thm:ell=0} we need just one more lemma.

\begin{lem}\label{lem:some-identity}
  Let $X$ and $Y$ be variables. For $m\in\N_0$,
  \begin{align*}
    \prod_{i=1}^m (1 + X^iY) &= \sum_{j=0}^m Y^jX^{\binom{j+1}{2}} \binom{m}{j}_X . 
  \end{align*}
\end{lem}

\begin{proof}
 Fix $j\in [m]_0$. Let $\mathcal{P}_{j,m-j}(n)$ be the set of
 partitions of $n$ with at most $j$ parts, each not larger than
 $m-j$. It is known that the $X^n$ coefficient of $\binom{m}{j}_X$ is
 $\#\mathcal{P}_{j,m-j}(n)$. Define $\mathcal{P}^*_{j,m-j}(n)\subseteq
 \mathcal{P}_{j,m-j}(n)$ to be the subset of partitions with distinct
 parts. There is a bijection $\mathcal{P}_{j,m-j}(n)\to
 \mathcal{P}^*_{j,m}(n+\binom{j+1}{2})$ given by $(a_1, \dots, a_j)
 \mapsto (a_1+j, \dots, a_j+1)$. Thus, the $X^{n+\binom{j+1}{2}}Y^j$
 coefficient of $\prod_{i=1}^m (1 + X^iY)$ is $\#\mathcal{P}_{j,m}^*(n
 + \binom{j+1}{2}) = \# \mathcal{P}_{j,m-j}(n)$, so the lemma follows.
\end{proof}
  
\begin{proof}[Proof of Theorem~\ref{thm:ell=0}]
  By Lemma~\ref{lem:no-S} it suffices to prove that
  \begin{align}\label{eqn:mat-E-S-1}
    q^{\binom{n-k+1}{2}} \binom{n}{k}_q \prod_{i=1}^{n-k}(1 + q^{k+i})  &= \sum_{b=0}^{n-k} q^{\binom{n-b-k+1}{2} + b(b+k+1)} \birk{n}{k}{b} .  
  \end{align}
  Applying Lemma~\ref{lem:birkhoff} and simplifying the expressions,
  \eqref{eqn:mat-E-S-1} is thus equivalent to 
  \begin{align*}
    \prod_{i=1}^{n-k} (1 + q^{k+i}) &= \sum_{b=0}^{n-k} q^{bk + \binom{b+1}{2}} \binom{n-k}{b}_{q} ,
  \end{align*}
  but this is just Lemma~\ref{lem:some-identity} with $m=n-k$, $X=q$ and
  $Y=q^k$.
\end{proof}

\subsubsection{An identity of differences}

In Proposition~\ref{prop:eqn.diff} we prove the
identity~\eqref{eqn:id-diff}.

\begin{lem}\label{lem:psi-omega-general}
  For $a,b\in \N_0$ with $a+b\leq n$ and $\ell\in\Z$,
  \begin{align*} 
    \psi_{n,\ell,\lri}^{\mathsf{C}} \left(\omega_{a,b}(\lri^n)\right) &= \birk{n-1}{a}{b} +
    q^{2n-a-2b-2\ell} \birk{n-1}{a}{b-1} \\ &\quad +
    q^{n-b-\ell}\left((1-q^{-a-1})\birk{n-1}{a+1}{b-1} +
    q^{-a}\birk{n-1}{a-1}{b}\right).
  \end{align*} 
\end{lem}

\begin{proof}
  It follows from~\eqref{eqn:omega_ab} that
  \begin{align} \label{psi.sum}
    \psi_{n,\ell,\lri}^{\mathsf{C}}\left(\omega_{a,b}(\lri^n)\right) &= \sum_{\Lambda
      \in\mathcal{L}_{a,b}(\lri^n)} \psi_{n,\ell,\lri}^{\mathsf{C}}\left(
    \prod_{i=1}^n(x_iq^{-i})^{\delta_i(\Lambda)}\right) =
    \sum_{\Lambda \in\mathcal{L}_{a,b}(\lri^n)} q^{-\ell
      \delta_n(\Lambda)}.
  \end{align}
  We decompose the sum in \eqref{psi.sum} according to the projections of the
  lattices $\Lambda$ onto a fixed sublattice $V_{(n-1)}$ of $\lri^n$ of rank
  $n-1$. Let $\Lambda'$ be the projection, and assume it has type $\mu$. There
  are only four possibilities of such a partition $\mu$, and in each case
  $\delta_n$ is determined. The fibers of the projections onto the lattices
  $\Lambda'$ are enumerated by \cite[Thm.~3.7]{MV/24}. We use the terminology
  and notation from \cite[Sec.~3]{MV/24}.

  We shorten $\left( 2^{(b)}, 1^{(a)},0^{(n-a-b)}\right)$ to $(b
  \mid a \mid n-a-b)$. The four types for $\mu$ are
  \begin{enumerate}
  \item $\mu = (b \mid a \mid n-a-b-1)$  \hfill (in which
    case~$\delta_n=0$),
  \item $\mu = (b \mid a-1 \mid n-a-b)$  \hfill ($\delta_n=1$),
  \item $\mu = (b-1 \mid a+1 \mid n-a-b-1)$  \hfill ($\delta_n=1$),
  \item $\mu = (b-1 \mid a \mid n-a-b)$ \hfill ($\delta_n=2$).
  \end{enumerate}
  Visually, these are the four ways to ``shave off'' a horizontal strip from
  $\lambda$.  The associated \emph{jigsaw} partitions
  $(\dual{\lambda},\dual{\mu})$ are therefore $\dual{\lambda} = (n-a-b \mid a
  \mid b)$ and 
  \begin{enumerate}
  \item $\widetilde{\mu} = (n-a-b-1 \mid a \mid b)$,
  \item $\widetilde{\mu} = (n-a-b \mid a-1 \mid b)$,
  \item $\widetilde{\mu} = (n-a-b-1 \mid a+1 \mid b-1)$,
  \item $\widetilde{\mu} = (n-a-b \mid a \mid b-1)$.
  \end{enumerate}
  See \cite[Eq.~(3.2)]{MV/24} for details regarding the jigsaw operation.

  The $\textup{gap}(\widetilde{\lambda},\widetilde{\mu})$ data are $0$, $n-a-b$,
  $n-b$, and $2n-a-2b$, respectively. Only the third case yields a non-trivial
  additional factor, namely $(1-q^{-a-1})$. Thus, the total contribution in each
  case is obtained by enumerating the lattices of type $\mu$ via
  Lemma~\ref{lem:birkhoff} and then applying \cite[Thm.~3.7]{MV/24}. The result
  follows because 
  \begin{align*}
    \psi_{n,\ell,\lri}^{\mathsf{C}}\left(\omega_{a,b}(\lri^n)\right) =\,& \birk{n-1}{a}{b}\cdot 1 \cdot q^{0\ell} + \\
    &\birk{n-1}{a-1}{b}\cdot q^{n-a-b} \cdot q^{-\ell} +\\
    &\birk{n-1}{a+1}{b-1}\cdot q^{n-b}(1-q^{-a-1}) \cdot q^{-\ell} +\\
    & \birk{n-1}{a}{b-1}\cdot q^{2n-a-2b} \cdot q^{-2\ell}.\qedhere
  \end{align*}
  \end{proof}

\begin{prop}\label{prop:eqn.diff}
  For $k\in[n]_0$ and $\ell\in\Z$,
  \begin{equation*}
    \psi_{n,\ell,\lri}^{\mathsf{C}}(\Omega_{n,\lri}^{\mathsf{C}}(\hopc{n}{k}{\lri})) -
    \psi_{n,\ell-1,\lri}^{\mathsf{C}}(\Omega_{n,\lri}^{\mathsf{C}}(\hopc{n}{k}{\lri})) = \Phi_{n,k,\ell}^{\mathsf{C}}(q)
    - \Phi_{n,k,\ell-1}^{\mathsf{C}}(q).
  \end{equation*}  
\end{prop}

\begin{proof}
  For $k=0$, apply Proposition~\ref{prop:k_0}, so assume $k\in [n]$. For
  $a,b\in\N_0$, we set 
  \begin{align*}
    S_{a,b,k}(\lri) &= q^{b(a+b+1)} (\#\mathrm{Sym}_{a,a-k}(\mathbb{F}_q)). 
  \end{align*}
  From Definition~\ref{def:Omega_k} it follows that
  \begin{equation}\label{eqn:psi-diff}
    \begin{split}
      &\psi_{n,\ell,\lri}^{\mathsf{C}}(\Omega_{n,\lri}^{\mathsf{C}}(\hopc{n}{k}{\lri})) - \psi_{n,\ell-1,\lri}^{\mathsf{C}}(\Omega_{n,\lri}^{\mathsf{C}}(\hopc{n}{k}{\lri})) \\
      &\quad = \sum_{a=k}^n\sum_{b=0}^{n-a} S_{a,b,k}(\lri) \left(q^{2\ell}\psi_{n,\ell,\lri}^{\mathsf{C}}(\omega_{a,b}(\lri^n)) - q^{2\ell-2}\psi_{n,\ell-1,\lri}^{\mathsf{C}}(\omega_{a,b}(\lri^n)) \right) . 
    \end{split}
  \end{equation}
  Lemma~\ref{lem:psi-omega-general} implies that 
  \begin{equation}\label{eqn:psi-omega-diff}
    \begin{split} 
      &q^{2\ell}\psi_{n,\ell,\lri}^{\mathsf{C}}(\omega_{a,b}(\lri^n)) -
      q^{2\ell-2}\psi_{n,\ell-1,\lri}^{\mathsf{C}}(\omega_{a,b}(\lri^n)) \\
      &\quad = q^{2\ell} (1 - q^{-2}) \birk{n-1}{a}{b} \\
      &\qquad + q^{n-b+\ell} (1 - q^{-1}) \left((1-q^{-a-1})\birk{n-1}{a+1}{b-1} +
      q^{-a}\birk{n-1}{a-1}{b}\right) .
    \end{split} 
  \end{equation}
  Putting \eqref{eqn:psi-diff} and \eqref{eqn:psi-omega-diff} together, we have
  \begin{align}\label{eqn:psi-diff-in-Wi}
    \psi_{n,\ell,\lri}^{\mathsf{C}}(\Omega_{n,\lri}^{\mathsf{C}}(\hopc{n}{k}{\lri})) - \psi_{n,\ell-1,\lri}^{\mathsf{C}}(\Omega_{n,\lri}^{\mathsf{C}}(\hopc{n}{k}{\lri})) &= q^{\ell}(1 - q^{-1})W_1 + q^{2\ell}(1 - q^{-2})W_2,
  \end{align}
  where 
  \begin{align*}
    W_1 &= \sum_{a=k}^n\sum_{b=0}^{n-a} S_{a,b,k}(\lri)q^{n-b}\left((1-q^{-a-1})\birk{n-1}{a+1}{b-1} +
    q^{-a}\birk{n-1}{a-1}{b}\right), \\
    W_2 &= \sum_{a=k}^n\sum_{b=0}^{n-a} S_{a,b,k}(\lri)\birk{n-1}{a}{b}.
  \end{align*}
  
  We express $W_1$ and $W_2$ in terms of
  $\Delta_{n,k}(q)$.  Because $\birk{n-1}{a}{b} = 0$
  whenever $a+b=n$, it follows from Corollary~\ref{cor:ell-0} that
  \begin{equation}\label{eqn:W2}
    \begin{split}
      W_2 &=
      \psi_{n-1,0,\lri}^{\mathsf{C}}(\Omega_{n-1,\lri}^{\mathsf{C}}(\hopc{n-1}{k}{\lri}))
      = \Delta_{n-1,k}(q) = \Delta_{n,k}(q)\dfrac{1 -
        q^{k-n}}{q^{2n}-1} .
    \end{split}
  \end{equation}
  Since $\psi_{n, 0,\lri}^{\mathsf{C}}(\omega_{a,b}(\lri^n)) = \birk{n}{a}{b}$ whenever
  $a+b\leq n$, Lemma~\ref{lem:psi-omega-general} implies that
  \begin{equation}\label{eqn:birk-n-n-1}
    \begin{split}
      \birk{n}{a}{b} &= \birk{n-1}{a}{b} +
      q^{2n-a-2b} \birk{n-1}{a}{b-1} \\ &\quad +
      q^{n-b}\left((1-q^{-a-1})\birk{n-1}{a+1}{b-1} +
      q^{-a}\birk{n-1}{a-1}{b}\right).
    \end{split}
  \end{equation}
  By Theorem~\ref{thm:ell=0}, \eqref{eqn:W2}, and
  \eqref{eqn:birk-n-n-1},
  \begin{align*}
    \Delta_{n, k}(q) - W_1 - W_2 &= \sum_{a=k}^n\sum_{b=0}^{n-a} S_{a,b,k}(\lri)q^{2n - a - 2b}\birk{n-1}{a}{b-1}.
  \end{align*}
  Since $S_{a,b,k}(\lri)q^{2n - a - 2b} = q^{2n}S_{a,b-1,k}(\lri)$, it follows that 
  \begin{equation}\label{eqn:W1-manip}
    \begin{split}
      \Delta_{n,k}(q) - W_1 - W_2 &= q^{2n} \Delta_{n-1, k}(q) = \Delta_{n,k}(q)\dfrac{q^{2n} - q^{n+k}}{q^{2n}-1} \\
      &= \Delta_{n,k}(q)\left(1 - \dfrac{q^{n+k} - 1}{q^{2n}-1} \right).
    \end{split}
  \end{equation}
  Hence, by \eqref{eqn:W2} and \eqref{eqn:W1-manip},
  \begin{align}\label{eqn:W1}
    W_1 &= \Delta_{n, k}(q) \dfrac{q^{n+k} - 2 + q^{k-n}}{q^{2n} - 1}.
  \end{align}

  Therefore, by \eqref{eqn:psi-diff-in-Wi}, \eqref{eqn:W2}, \eqref{eqn:W1} and Definition~\ref{def:Phi},
  \begin{align*}
    &\psi_{n,\ell,\lri}^{\mathsf{C}}(\Omega_{n,\lri}^{\mathsf{C}}(\hopc{n}{k}{\lri})) - \psi_{n,\ell-1,\lri}^{\mathsf{C}}(\Omega_{n,\lri}^{\mathsf{C}}(\hopc{n}{k}{\lri})) \\
    &\quad = q^{\ell}(1 - q^{-1})\Delta_{n,k}(q)\dfrac{q^{n+k} - 2 + q^{k-n}}{q^{2n}-1} + q^{2\ell}(1 - q^{-2})\Delta_{n,k}(q)\dfrac{1 - q^{k-n}}{q^{2n}-1} \\
    &\quad = \Phi_{n,k,\ell}^{\mathsf{C}}(q) - \Phi_{n,k,\ell-1}^{\mathsf{C}}(q). \qedhere
  \end{align*}
\end{proof}

\begin{proof}[Proof of Theorem~\ref{thm:Ehr-Sat}]
  Combine Corollary~\ref{cor:ell-0} and
  Proposition~\ref{prop:eqn.diff}.
\end{proof}

\subsection{Proof of Theorem~\ref{thmabc:ehr.sat}}
\label{sec:proof.thmabc.ehr.sat}

Apply Proposition~\ref{prop:part-1} and Theorem~\ref{thm:Ehr-Sat}. The
polynomiality of $\Phi_{n,k,\ell}^{\mathsf{C}}(Y)$ follows from
Definition~\ref{def:Omega_k}. A simple computation using
Definition~\ref{def:Phi} proves the final claim. \qed

\subsection{Spherical functions and averaged Ehrhart coefficients}\label{subsec:cor.spher}

Theorem~\ref{thmabc:ehr.sat} allows us to interpret the spherical
function $\omega_{n,\ell,p}$, defined in~\eqref{def:spher}, in terms of Ehrhart coefficients.

\begin{prop}\label{prop:spherical}
  Let $\ell\in [2n]_0$ and $P\in\mathscr{P}^{\Lambda_0}$ be a
  polytope with $\mathscr{E}_{2n,\ell}(P)\neq 0$. For all $k\in
  [n]_0$,
  \begin{align*}
    \omega_{n,\ell,p}(D_{n,k,p}^{-1}) &= \dfrac{\hopc{n}{k}{p} \mathscr{E}_{2n,\ell}(P)}{(\#\mathscr{D}_{n,k,p}^{\mathsf{C}})\mathscr{E}_{2n,\ell}(P)} .
  \end{align*}
\end{prop}

\begin{proof}
  Define a homomorphism (see \cite[Prop.~1.2.6]{Macdonald/71}) $\widehat{\omega}_{n,\ell,p} : \mcH_{n,p}^{\mathsf{C}}\otimes_{\Z}\C \longrightarrow \C$ by 
  \begin{align}\label{eqn:Hecke-int}
    T &\longmapsto \int_{\GSp_{2n}(\Q_p)} T(g)\omega_{n,\ell,p}(g^{-1}) \,\mathrm{d}\mu.
  \end{align}
  Since $\mathscr{E}_{2n,\ell}(P)\neq 0$, it follows from
  Theorem~\ref{thmabc:ehr.sat} that for all $k\in [n]_0$
  \begin{align}\label{eqn:Hecke-quo}
    \widehat{\omega}_{n,\ell,p}(\hopc{n}{k}{p}) &= \psi_{n,\ell,p}^{\mathsf{C}}(\Omega_{n,p}^{\mathsf{C}}(\hopc{n}{k}{p})) = \dfrac{\hopc{n}{k}{p} \mathscr{E}_{2n,\ell}(P)}{\mathscr{E}_{2n,\ell}(P)} .
  \end{align}
  Since $\hopc{n}{k}{p}$ is the characteristic function of the $\Gamma_n$-double
  coset $\Gamma_nD_{n,k,p}\Gamma_n$ and $\omega_{n,\ell,p}$ is bi-invariant on
  $\Gamma_n$-double cosets, \eqref{eqn:Hecke-int} and~\eqref{eqn:Hecke-quo}
  imply
  \begin{align*}
    \omega_{n,\ell,p}(D_{n,k,p}^{-1}) &= \dfrac{\hopc{n}{k}{p} \mathscr{E}_{2n,\ell}(P)}{\mu(\Gamma_nD_{n,k,p}\Gamma_n)\mathscr{E}_{2n,\ell}(P)} . 
  \end{align*}
  Since $\mu(\Gamma_nD_{n,k,p}\Gamma_n) = \#\mathscr{D}_{n,k,p}^{\mathsf{C}}$,
  the statement follows.
\end{proof}

Since $\left|\det(g)\right|_{p}^s = |\lattice{g} : \Lambda_0|^{-s}$,
Proposition~\ref{prop:spherical} implies that, for $\ell\in[2n]_0$ and
$P\in\poly{\Lambda_0}$ with $\mathscr{E}_{2n,\ell}(P)\neq 0$,
\begin{align}\label{eqn:local-interpret}
  \lZC{n}{\ell}{p}(s) &= \sum_{\Gamma_ng \in \Gamma_n\setminus \Gamma_n\mathrm{G}^+_n(\Q_p)\Gamma_n} \dfrac{\mathscr{E}_{2n,\ell}(g\cdot P)}{\mathscr{E}_{2n,\ell}(P)} \left|\lattice{g} : \Lambda_0\right|^{-s}. 
\end{align} 
Hence, $\lZC{n}{\ell}{p}(s)$ encodes normalized sums of the $\ell$th
coefficients of the Ehrhart polynomial of~$P$. 

\section{Explicit formulae for Ehrhart--Hecke zeta functions: proof of Theorem~\ref{thmabc:BIgu}}
\label{sec:HS} 

Theorem~\ref{thmabc:BIgu} gives explicit combinatorial formulae for
the local Ehrhart--Hecke zeta functions $\lZC{n}{\ell}{\lri}$. Our
proof of this theorem in Section~\ref{subsec:proof.thmabc:BIgu}
leverages a formula for a coarsening of the \emph{Hermite--Smith
series}, enumerating lattices simultaneously by their Smith normal
form and last Hermite parameter; see Section~\ref{subsec:HS}.

\subsection{Hermite--Smith series}\label{subsec:HS}

We develop additional notation only used in this section,
matching~\cite{MV/24}. Recall the lattice definitions from
Section~\ref{sec:lattices}. We set $V=\lri^n$ and write
$\mathcal{L}(V)$ for the set of finite-index sublattices of $V$. For a
partition $\lambda=(\lambda_1,\dots,\lambda_n)\in\mcP_n$, we write
\begin{align*}
  \Dif(\lambda) &= (\lambda_1-\lambda_2,\dots,\lambda_{n-1}-\lambda_n,
  \lambda_n) = \left(\Dif_i(\lambda)\right)_{i\in [n]}\in\N_0^n
\end{align*}
for the integer composition comprising the differences of the parts of
$\lambda$.

Let $\host_n \subset \mathcal{P}_n\times\mathcal{P}_{n-1}$ be the pairs of
partitions $(\lambda,\mu)$ such that $\lambda_i\geq \mu_i$ for all $i\in [n-1]$
and the skew diagram $\lambda-\mu$ is a horizontal strip. Given a lattice
$\Lambda\in \mcL(V)$ of type $\lambda$, it follows from \cite[Lem.~3.14]{MV/24}
that $(\lambda, \lambda_{(n-1)}(\Lambda))\in \host_n$. For
$(\lambda,\mu)\in\host_n$, we define
\[ 
  \mcE_{\lambda,\mu}(V) = \left\{ \Lambda \in \mcL(V) \mid \Lambda \text{ of type }\lambda\text{ and }\varpi_{n-1}(\Lambda) \text{ of type }\mu  \right\} .
\]
Understanding the cardinality of $\mcE_{\lambda,\mu}(V)$ is an
important ingredient in our proof of Theorem~\ref{thmabc:BIgu}. To
this end we associate with a pair $(\lambda,\mu)\in \host_n$ two sets:
\begin{equation*}
  \begin{split}
  \mathcal{I}(\lambda,\mu) &= \left\{ \lambda'_a ~\middle|~ a \in \N,
  \lambda'_a = \mu'_a+1 > 0 \right\}\subseteq
          [n],\\ \mathcal{J}(\lambda,\mu) &= \left\{ \lambda'_a + 1
          ~\middle|~ a \in \N, \lambda'_a = \mu'_a > 0 \right\}
          \subseteq [n].
  \end{split}
\end{equation*}

\begin{prop}\label{prop:E-card} 
  For $(\lambda,\mu) \in \host_n$, with $I=\mathcal{I}(\lambda,\mu)$ and
  $J=\mathcal{J}(\lambda,\mu)$,
  \begin{align}\label{eqn:ext-card}
    \#\mcE_{\lambda,\mu}(V) &= \Psi_{n, I, J}\left(q^{-1}\right) \prod_{\substack{a\in \N \\ \lambda_a'\neq \mu_a'}} q^{\lambda_a'(n-\lambda_a')} \prod_{\substack{b\in\N \\ \lambda_b'=\mu_b'}} q^{\mu_b'(n - 1 - \mu_b')} .
  \end{align}
\end{prop}

\begin{proof}
  A lattice $\Lambda\in \mcE_{\lambda,\mu}(V)$ is an extension of a
  lattice $\Lambda'\in\mcL(V_{(n-1)})$ of type $\mu$ by a cyclic
  $\lri$-module.  Let
  \begin{equation}\label{fac1}
    N_{n-1,\mu}(\lri) = \# \left\{ \Lambda' \in
    \mcL(V_{(n-1)}) ~\middle|~ \Lambda' \text{ of type }
    \mu\right\} .
  \end{equation} 
  For $\Lambda' \in \mcL(V_{(n-1)})$, we define
  \begin{equation}\label{fac2}
    \mcEpr_{\lambda}(V,\Lambda') = \left\{ \Lambda \in\mcL(V)
    ~\middle|~ \varpi_{n-1}(\Lambda) = \Lambda',\ \Lambda \text{ of type }
    \lambda \right\}.
  \end{equation}
  The number $\#\mcE_{\lambda,\mu}(V)$ clearly factors as
  \begin{equation*}
    \#\mcE_{\lambda,\mu}(V) = (\# \mcE^{\textup{pr}}_{\lambda}(V,
    \Lambda')) N_{n-1,\mu}(\lri).
  \end{equation*}
  
  For the first factor \eqref{fac2}, we apply \cite[Thm.~3.17 \&
    Lem.~4.4]{MV/24} and obtain
  \begin{equation*}
    \#\mcEpr_{\lambda}(V, \Lambda') = \prod_{\substack{a\in\N
        \\ \lambda_a' \neq \mu_a'}}q^{n-\lambda_a'} \prod_{b\in
      R_{\lambda,\mu}}\left(1-q^{-\Dif_b(\mu')}\right),
  \end{equation*}
  where $R_{\lambda, \mu} = \{a\in \N ~|~ \mu_a'=\lambda_a',\ \mu_{a+1}' <
  \lambda_{a+1}'\}$. To get the second factor~\eqref{fac1} we have $\lambda'_a -
  \mu'_a \in \{0,1 \}$ for all $a\in \N$ since $\lambda - \mu$ is a horizontal
  strip. Note that $(I-1)\cup (J-1) = \{\mu_a' ~|~ a\in\N\} \subseteq[n-1]$ is
  the set of distinct parts of $\mu'$; in particular it is independent
  of~$\lambda$. From Lemma~\ref{lem:Birkhoff-formula} it follows that
  \begin{equation}\label{fac1.re}
    N_{n-1,\mu}(\lri) = \binom{n-1}{(I-1)\cup
      (J-1)}_{q^{-1}}~\prod_{b\in \N}q^{(n-1-\mu'_b)\mu'_b}.
  \end{equation}

  To get the $q$-powers in~\eqref{eqn:ext-card}, note that 
  \begin{multline*}
    \sum_{\substack{a\in \N\\ \lambda_a' \neq \mu_a'}}(n-\lambda_a') + \sum_{b\in \N} (n-1-\mu_b')\mu_b' \\= \sum_{\substack{a\in \N\\ \lambda_a' \neq \mu_a'}}\left((n-\lambda_a') + (n-1-\mu_a')\mu_a'\right)  + \sum_{\substack{b\in \N \\ \lambda_b' = \mu_b'}} (n-1-\mu_b')\mu_b' \\ 
    = \sum_{\substack{a\in \N\\ \lambda_a' \neq \mu_a'}}(n-\lambda_a')\lambda_a' + \sum_{\substack{b\in \N \\ \lambda_b' = \mu_b'}} (n-1-\mu_b')\mu_b'.
  \end{multline*}
  To get the $\Psi_{n,I,J}$-factor in~\eqref{eqn:ext-card}, it suffices to show
  that 
  \begin{align}\label{eqn:MV2-to-covers}
    \prod_{b\in R_{\lambda,\mu}}\left(1-Y^{\Dif_b(\mu')}\right) &= \prod_{(i,j) \in C_{I,J-1}} \left(1 - Y^{j-i+1}\right) . 
  \end{align}

  We define $\rho : R_{\lambda,\mu} \to C_{I, J-1}$ via $a \mapsto (\lambda_{a+1}',\mu_a')$.
  First, we show that $\rho$ is well defined. Let $a \in
  R_{\lambda,\mu}$, so $\lambda_{a+1}' = \mu_{a+1}' + 1$. Hence, $\lambda_{a+1}'
  \in I$ and $\mu_a' + 1 \in J$. By construction, $(\lambda_{a+1}', \mu_a')$ is
  a cover in the poset $P_{I, J-1}$, so $\rho$ is well defined. Clearly $\rho$
  is an injection, so we show that $\rho$ is a surjection. Let
  $(\lambda_a',\lambda_b')\in C_{I,J-1}$, so $\lambda_a'=\mu_a'+1$ and
  $\lambda_b'=\mu_b'$. Since $(\lambda_b',0)$ covers $(\lambda_a',1)$ in
  $P_{I,J-1}$, we have $\lambda_b'>\lambda_a'$. Thus, there exists $c\in \N$
  such that $\lambda_b'=\mu_b'=\lambda_c'$ and $\lambda_a'=\lambda_{c+1}'$, so
  $\rho(c) = (\lambda_a', \lambda_b')$ and $\rho$ is a bijection. \eqref{eqn:MV2-to-covers} holds as, lastly, for $b\in R_{\lambda,\mu}$, 
  \begin{align*}
    \Dif_b(\mu') &= \mu_b'-\mu_{b+1}' = \mu_b' - \lambda_{b+1}' + 1. \qedhere
  \end{align*}
\end{proof}

Proposition~\ref{prop:E-card} motivates the definition of the
following map:
\begin{equation*}
  W_n : \host_n \longrightarrow 2^{[n]} \times 2^{[n]} \quad (\lambda,\mu) \longmapsto (\mathcal{I}(\lambda,\mu), \mathcal{J}(\lambda,\mu)).
\end{equation*}
Clearly $(I,J) \in W_n(\host_n)$ if and only if $1\notin J$. The
following lemma shows that the fibers of $W_n$ are obtained by
horizontally ``stretching'' a unique minimal pair
$(\lambda,\mu)\in\host_n$.

\begin{lem}\label{lem:Wn-fiber}
  For $(I,J)\in W_n(\host_n)$, there exists a unique pair $(\lambda,\mu)\in \host_n$
  with the following property: $W_n(\nu,\rho) = (I, J)$ if and only if  
  \begin{align*}
    \{ \nu_a' ~|~ a\in \N \} &= \{ \lambda_a' ~|~ a\in\N\}, & 
    \{ \rho_a' ~|~ a\in \N \} &= \{ \mu_a' ~|~ a\in\N\}.
  \end{align*}
  Additionally, the fiber $W_n^{-1}(I,J)$ is in bijection with $\N^I\times
  \N^J$.
\end{lem}

\begin{proof}
  The reverse direction of the first claim is clear, so we prove the
  forward direction. Define multisets $M_1 = I \cup (J - 1)$ and $M_2
  = (I - 1) \cup (J - 1)$. Note that elements have multiplicity at
  most two. Let $\lambda$ and $\mu$ be the partitions obtained from
  $M_1$ and $M_2$, respectively, by sorting the elements in decreasing
  order. Their conjugate $(\lambda',\mu')\in\host_n$ is our desired
  pair.

  To establish the second claim, we fix multiplicities $m_I\in \N^I$
  and $m_J\in \N^J$. Define multisets $M_1'$ and $M_2'$ similar to
  $M_1$ and $M_2$ but with multiplicities given by $m_I$ and $m_J$,
  respectively. More specifically, the multiplicity of $a\in [n]$ in
  $M_1'$ is $m_I(a) + m_J(a+1)$ and similarly for $M_2'$. By the same
  procedure, we obtain two partitions contained in $W_n^{-1}(I, J)$.
\end{proof}

Let $\bm{x}=(x_1, \dots, x_n)$ and $\bm{y} = (y_1, \dots, y_n)$ be
variables.  The \define{Hermite--Smith series} defined in
\cite{AMV_fpsac/24} enumerates finite-index sublattices of $\lri^n$
simultaneously by their types and Hermite compositions:
\begin{align}\label{def:HS}
  \HS_{n,\lri}(\bm{x}, \bm{y}) &= \sum_{\Lambda \in \mcL(V)}
  \bm{x}^{\Dif(\lambda({\Lambda}))}\bm{y}^{\Hermite{\Lambda}} =
  \sum_{\Lambda \in \mcL(V)}
  \prod_{i=1}^{n}x_i^{\Dif_i(\lambda(\Lambda))}y_i^{\delta_i(\Lambda)}.
\end{align}

We express the local Ehrhart--Hecke zeta functions
$\lZC{n}{\ell}{\lri}$ in terms of a specific coarsening of the
multivariate rational functions $\HS_{n,\lri}$. Specifically, we ignore all but
the last entry $\delta_n$ of the Hermite composition. To this end we define
\begin{equation}\label{def:HSbar}
  \HSsp_{n,\lri}(\bm{x}, y) = \HS_{n,\lri}(\bm{x}, 1, \dots, 1,
  y).
\end{equation}
In Theorem~\ref{thm:HS.explicit} we provide an explicit formula for
$\HSsp_{n,\lri}(\bm{x},y)$, far simpler than the specialization of
\cite[Thm.~D]{MV/24}. The latter gives an expression that involves a
summation over semistandard Young tableaux with labels from $[n]$ and
without repeated columns, featuring roughly $2^{2^{n}}$ summands.

\begin{thm}\label{thm:HS.explicit}
  For all cDVR $\lri$ with residue field cardinality $q$, 
  \begin{align*}
    \overline{\mathrm{HS}}_{n,\lri}(\bm{x}, y) &=
    \sum_{\substack{I\subseteq [n] \\ J\subseteq [n-1]}} \Psi_{n, I, J+1}(q^{-1}) \prod_{i\in I}
    \dfrac{q^{i(n-i)}x_iy}{1-q^{i(n-i)}x_iy} \prod_{j\in J}
    \dfrac{q^{j(n-1-j)}x_j}{1-q^{j(n-1-j)}x_j}.
  \end{align*}
\end{thm}

\begin{proof}
  From~\eqref{def:HS} and~\eqref{def:HSbar} it follows that
  \begin{equation}\label{eqn.HS.start}
    \overline{\HS}_{n,\lri}(\bfx,y) = \sum_{\Lambda\in \mcL(V)}
    \bfx^{\Dif(\lambda(\Lambda))}y^{\delta_n(\Lambda)}.
   \end{equation} 
  Recall that $q^{\delta_n(\Lambda)}$ is the index of the intersection
  of $\Lambda$ with a cyclic summand of~$\lri^n$. We therefore rewrite
  \eqref{eqn.HS.start} as
  \begin{equation}\label{equ:HS.rew}
    \overline{\HS}_{n,\lri}(\bfx,y) =
    \sum_{(\lambda,\mu)\in\host_n} (\#\mcE_{\lambda,\mu}(V))\bfx^{\Dif(\lambda)}y^{|\lambda|-|\mu|}.
  \end{equation}

  Let $(I,J)\in W_n(\host_n)$, and let $(m_I,m_J)\in \N^I\times \N^J$. By
  Lemma~\ref{lem:Wn-fiber}, there is a unique pair $(\lambda, \mu)\in
  W_n^{-1}(I,J)$ associated with $(m_I,m_J)$. It follows that
  \[ 
    \Dif_a(\lambda) = m_I(a) + m_J(a-1)
  \]
  for $a\in [n]$.
  Moreover $|\lambda|-|\mu| = \sum_{i\in I}m_I(i)$. Write $\widehat{\Psi}_{n,\lambda,\mu}(Y) =
  \Psi_{n, \mathcal{I}(\lambda,\mu), \mathcal{J}(\lambda,\mu)}(Y)$. By
  Proposition~\ref{prop:E-card} and~\eqref{equ:HS.rew} we have
  {\small
  \begin{align*}
    \overline{\HS}_{n,\lri}(\bfx,y) &= \sum_{(\lambda,\mu)\in \host_n}\widehat{\Psi}_{n, \lambda,\mu}\left(q^{-1}\right)  \bfx^{\Dif(\lambda)}y^{|\lambda|-|\mu|} \prod_{\substack{a\in \N \\ \lambda_a'\neq \mu_a'}} q^{\lambda_a'(n-\lambda_a')} \prod_{\substack{b\in\N \\ \lambda_b'=\mu_b'}} q^{\mu_b'(n - 1 - \mu_b')} \\ 
    &= \sum_{\substack{I\subseteq [n] \\ J\subseteq [n-1]}} \Psi_{n, I, J+1}(q^{-1}) \sum_{\substack{m_I\in \N^I\\ m_J\in \N^J}} \prod_{i\in I} \left(q^{i(n-i)}x_iy\right)^{m_I(i)} \prod_{j\in J} \left(q^{j(n-1-j)}x_j\right)^{m_J(j)} \\
    &= \sum_{\substack{I\subseteq [n] \\ J\subseteq [n-1]}} \Psi_{n, I, J+1}(q^{-1}) \prod_{i\in I} \frac{q^{i(n-i)}x_iy}{1-q^{i(n-i)}x_iy} \prod_{j\in J} \frac{q^{j(n-1-j)}x_j}{1-q^{j(n-1-j)}x_j} . \qedhere
  \end{align*}
  }
\end{proof}

\subsection{Proof of Theorem~\ref{thmabc:BIgu}}\label{subsec:proof.thmabc:BIgu}

We prove the theorem by relating $\lZC{n}{\ell}{\lri}$ and
$\overline{\HS}_{n,\lri}$ to both the symplectic Hecke series and the
Hall--Littlewood--Schubert series of \cite{MV/24}. For a
statement~$\sta$, we denote by $[\sta]\in \{0,1\}$ the Iverson bracket
which is $1$ if and only if $\sta$ is true. From \cite[Thm.~D]{MV/24}
it follows that
\begin{align*}
  \overline{\HS}_{n,\lri} \left(\left(q^{\binom{i+1}{2}+\ell}t^n\right)_{i\in
    [n]},q^{-\ell}\right) &= \HLS_n\left(q^{-1}, \left(q^{d_n(C) +
    \binom{\#C+1}{2} + \ell [1 \notin C]}t^n\right)_C\right),
\end{align*}
where $d_n(C)$ is the dimension of the Schubert variety associated
with $C$ in the Grassmannian $\mathrm{Gr}(\#C, n)$;
see~\cite[Eq.~(1.5)]{MV/24}.  We write the symplectic Hecke series
defined in~\cite[Eq.~(1.7)]{MV/24} as $\Hecke_{n,\lri}(\bfx, x_0X) \in
\Q(x_0, \bfx, X)$.  By Theorem~\ref{thmabc:ehr.sat},
\begin{align}\label{eqn:Z=Hecke}
  {\lZC{n}{\ell}{\lri}(s)} &= \Hecke_{n,\lri}(q^1,q^2,\dots, q^{n-1}, q^{n-\ell},q^{\ell}t^n).
\end{align}
Thus, by~\eqref{eqn:Z=Hecke} and~\cite[Thm.~E]{MV/24} we find that
\begin{align*}
  {\left(1 - q^{\ell}t^n\right)}
  {\lZC{n}{\ell}{\lri}(s)}   &= {\left(1 - q^{\ell}t^n\right)} \Hecke_{n,\lri}(q^1,q^2,\dots, q^{n-1}, q^{n-\ell},q^{\ell}t^n)\\
  &= \HLS_n\left(q^{-1}, \left(q^{d_n([n]\setminus \widehat{C}) + \binom{\#\widehat{C} + 1}{2} + \ell [1 \not\in \widehat{C}]}t^n\right)_C\right),
\end{align*}
where $\widehat{C} = \{n - i + 1 ~|~ i\in C\}$. Because $d_n(C) +
d_n([n]/C) = \# C (n - \#C)$, it follows that $d_n(C) =
d_n([n]\setminus \widehat{C})$.  It thus suffices to show that {\small
\begin{align*}
  \HLS_n\left(q^{-1}, \left(q^{d_n(C) + \binom{\#C+1}{2} + \ell [1 \notin C]}t^n\right)_C\right) &= \HLS_n\left(q^{-1}, \left(q^{d_n([n]\setminus \widehat{C}) + \binom{\#\widehat{C} + 1}{2} + \ell [1 \not\in \widehat{C}]}t^n\right)_C\right).
\end{align*}
}
But this is equivalent to the identity
\begin{equation*}
  \Hecke_{n,\lri}(q^{n-\ell},q^{n-1},\dots,q^2,q^1,\, q^{\ell}t^n) = \Hecke_{n,\lri}(q^1,q^2,\dots, q^{n-1}, q^{n-\ell},\, q^{\ell}t^n),
\end{equation*}
which is a consequence of the symmetry of $\Hecke_{n,p}(\bfx,X)$ in the variables $\bfx$. Hence,
\begin{align}\label{eqn:Z=HS}
  {\lZC{n}{\ell}{\lri}(s)} &= \dfrac{1}{1 - q^{\ell}t^n} \overline{\HS}_{n,\lri} \left(\left(q^{\binom{i+1}{2}+\ell}t^n\right)_{i\in [n]},q^{-\ell}\right).
\end{align}

For $I\subseteq [n]$ and $J\subseteq [n-1]$, we have $\Psi_{n, I,
  (J+1)\cup\{1\}}(Y) = \Psi_{n, I, J+1}(Y)$. We may thus
rewrite~\eqref{eqn:Z=HS} using Theorem~\ref{thm:HS.explicit} as
\begin{align*}
  {\lZC{n}{\ell}{\lri}(s)} 
  &= \!\!\sum_{\substack{I\subseteq [n] \\ J\subseteq [n-1]_0}} \!\!\Psi_{n, I, J+1}(q^{-1}) \prod_{i\in I}
  \dfrac{q^{\binom{i+1}{2} + i(n-i)}t^n}{1-q^{\binom{i+1}{2} + i(n-i)}t^n} \prod_{j\in J}
  \dfrac{q^{\binom{j+1}{2} + j(n-1-j) + \ell}t^n}{1-q^{\binom{j+1}{2} + j(n-1-j) + \ell}t^n} \\ 
  &= \sum_{I,J \subseteq [n]} \Psi_{n, I, J}(q^{-1}) \prod_{i\in I}
  \dfrac{q^{\binom{i+1}{2} + i(n-i)}t^n}{1-q^{\binom{i+1}{2} + i(n-i)}t^n} \prod_{j\in J}
  \dfrac{q^{\binom{j}{2} + (j-1)(n-j) + \ell}t^n}{1-q^{\binom{j}{2} + (j-1)(n-j) + \ell}t^n} . 
\end{align*}
Writing $\binom{j}{2} + (j-1)(n-j)$ as $\binom{j+1}{2} + j(n-j) - n$ yields the theorem. \qed

\subsection{Proof of Corollary~\ref{corabc:refl}}\label{sec:reflection-proof}

By applying Theorem~\ref{thmabc:BIgu} to both
$\lZC{n}{2n-\ell}{\lri}(s)$ and
$\lZC{n}{\ell}{\lri}(s-\frac{n-\ell}{n})$, one sees that
Corollary~\ref{corabc:refl} would follow if $\Psi_{n,I,J}(Y)$ and
$\Psi_{n,J,I}(Y)$ coincided for all $I, J\subseteq [n]$. However, this
is false in general due to the asymmetry in the covering set $C_{I,
  J-1}$.  Instead, we proceed by rewriting the formula given in
Theorem~\ref{thmabc:BIgu} over a common denominator; see
Lemma~\ref{lem:com.den}.

To this end, we establish notation that will also be relevant in
Section~\ref{subsec:analytic}. Define
\begin{align}
  \Theta_{n,I,J}(Y) &= \sum_{A\subseteq I}\sum_{B\subseteq J} (-1)^{\#(I\setminus A) + \#(J\setminus B)} \Psi_{n, A, B}(Y) \in \Z[Y], \label{eqn:Theta-def}\\
  \beta(I) &= \sum_{i\in I}\binom{i+1}{2} + i(n-i) \in \N_0. \label{eqn:beta-def}
\end{align}
For $k\in [n]_0$, we write $\binom{[n]}{k}$ for the set of $k$-element
subsets of $[n]$. The following lemma is a straightforward consequence
of Theorem~\ref{thmabc:BIgu}.

\begin{lem}\label{lem:com.den}
  For all $\ell\in \Z$,
  \[ 
  \lZC{n}{\ell}{\lri}(s) =
  \dfrac{\sum_{m=0}^{2n}\sum_{k=0}^m\sum_{I\in\binom{[n]}{k}}\sum_{J\in\binom{[n]}{m-k}}\Theta_{n,I,J}(q^{-1})
    q^{\beta(I)+\beta(J) + (m-k)(\ell - n)}t^{mn}}{\prod_{i\in [n]}(1 -
    q^{\beta(\{i\})}t^n)\prod_{j\in [n]}(1 - q^{\beta(\{j\}) - n +
      \ell}t^n)} .
  \]
\end{lem}

\begin{proof}[Proof of Corollary~\ref{corabc:refl}]
  We see that both $\lZC{n}{2n-\ell}{\lri}(s)$ and
  $\lZC{n}{\ell}{\lri}(s-\frac{n-\ell}{n})$ have the same denominator
  by Lemma~\ref{lem:com.den}. The two respective numerators are
  \begin{align*}
    \sum_{m=0}^{2n}\sum_{k=0}^m\sum_{I\in\binom{[n]}{k}}\sum_{J\in\binom{[n]}{m-k}}\Theta_{n,I,J}(q^{-1})
    q^{\beta(I)+\beta(J) + (m-k)(n-\ell)}t^{mn}
  \end{align*}
  and 
  \begin{align*}
    \sum_{m=0}^{2n}\sum_{k=0}^m\sum_{I\in\binom{[n]}{k}}\sum_{J\in\binom{[n]}{m-k}}\Theta_{n,I,J}(q^{-1})
    q^{\beta(I)+\beta(J) + k(n - \ell)}t^{mn}, 
  \end{align*}
  which are also equal. 
\end{proof}

\section{The type-$\mathsf{A}$ story}
\label{sec:type-A}

We paraphrase some of the the results of Gunnells and Rodriguez
Villegas~\cite{GRV/07} from our perspective. Within this section, we
redefine $\mathrm{G}_n^+(K)$ and $\Gamma_n$ to be the type
$\mathsf{A}$ analogues from~\eqref{eqn:Gamma-G+}: set
$\mathrm{G}_n^{+}(K) = \GL_n(K)\cap \Mat_n(\lri)$ and $\Gamma_{n} =
\GL_n(\lri)$. Thus, $(\mathrm{G}_n^+(K), \Gamma_n)$ is a Hecke pair
with associated Hecke ring~$\mcH_{n,\lri}^{\mathsf{A}}$.  For $k\in
[n]$ let
\begin{align*}
  D_{n,k,\lri}^{\mathsf{A}} = \diag\left(1^{(n-k)}, \pi^{(k)}\right)\in\Mat_n(\lri),
\end{align*}
and let $\hopa{n}{k}{\lri} \in \mcH_{n,\lri}^{\mathsf{A}}$ be the
characteristic function of
$\Gamma_nD_{n,k,\lri}^{\mathsf{A}}\Gamma_n$. We have
$\mcH_{n,\lri}^{\mathsf{A}}= \Z[\hopa{n}{1}{\lri},\dots,
  \hopa{n}{n}{\lri}]$; see e.g.\ \cite[(2.7) in
  Sec.~V.2]{Macdonald/71}. Similar to \eqref{def:Dnk}, define
\begin{align*}
  \mathscr{D}_{n,k,\lri}^{\mathsf{A}} = \Gamma_n\setminus \Gamma_n D_{n,k,\lri}^{\mathsf{A}}\Gamma_n.
\end{align*}
In analogy to \eqref{eqn:HeckeC-action} we define an action of
$\mcH_{n,p}^{\mathsf{A}}$ on $\UIV_n$, differing from~\cite{GRV/07}
only cosmetically.

Let $\Omega_{n,\lri}^{\mathsf{A}}: \mcH_{n,\lri}^{\mathsf{A}} \to
\Q[\bm{x}]^{S_{n}}$ be the type-$\mathsf{A}$ Satake isomorphism, where $S_{n}$
is the symmetric group of degree $n$. For $\ell\in\Z$, we define a homomorphism
\begin{align*}
  \psi_{n,\ell,\lri}^{\mathsf{A}} : \Q[\bm{x}]\longrightarrow \Q \quad\quad x_i \longmapsto \begin{cases}
    q^{i} & \text{ if } i \in [n-1], \\ 
    q^{\ell} & \text{ if } i = n.
  \end{cases}
\end{align*}
Let $e_{n,k}(\bm{x})\in \Z[\bm{x}]$ be the homogeneous elementary
symmetric polynomial of degree~$k$, i.e.\ the sum of all products of
$k$ distinct variables.  We define
\begin{align*}
  \Phi_{n,k,\ell}^{\mathsf{A}}(Y) &= Y^k\binom{n-1}{k}_Y + Y^{\ell} \binom{n-1}{k - 1}_Y \in \Z[Y].
\end{align*}

\begin{thm}[{Gunnells and Rodriguez Villegas~\cite{GRV/07}}]\label{thm:A}
  Let $k\in [n]$ and $\ell\in [n]_0$. For all primes $p$,
  \begin{align*} 
    \hopa{n}{k}{p}\mathscr{E}_{n,\ell} &= \Phi_{n,k,\ell}^{\mathsf{A}}(p) \mathscr{E}_{n,\ell}, & 
    \Phi_{n,k,\ell}^{\mathsf{A}}(p) &= \psi_{n,\ell,p}^{\mathsf{A}}(\Omega_{n,p}^{\mathsf{A}}(\hopa{n}{k}{p})).
  \end{align*}
  Additionally, the polynomials $\Phi_{n,k,\ell}^{\mathsf{A}}(Y)$ satisfy
  \begin{align*}
    \dfrac{\Phi_{n,k,\ell}^{\mathsf{A}}(Y)}{\Phi_{n,n-k,n-\ell}^{\mathsf{A}}(Y)} &= Y^{k+\ell-n} . 
  \end{align*}
\end{thm}

\begin{proof} 
  The first and third equalities follow
  from~\cite[Thm.~1.4]{GRV/07}. The second equality can be deduced
  from the discussion in~\cite[Sec.~1.10]{GRV/07}; we provide some
  details here. Note that $\binom{n}{k} = p^{\binom{k}{2}}\,
  e_{n,k}(1, p, \dots, p^{n-1})$. By \cite[Lem.~3.2.21]{Andrianov/87},
  $\Omega_{n,p}^{\mathsf{A}}(\hopa{n}{k}{p}) = p^{-\binom{k}{2}}\,
  e_{n,k}$, so
  \begin{align*} 
    \Phi_{n,k,\ell}^{\mathsf{A}}(p) &= p^k e_{n-1,k}(1, p, \dots, p^{n-2}) + p^{\ell} e_{n-1,k-1}(1, p, \dots, p^{n-2})\\
    &= p^{-\binom{k}{2}}\psi_{n,\ell,p}^{\mathsf{A}}(e_{n-1,k} + x_n e_{n-1,k-1}) \\
    &= p^{-\binom{k}{2}}\psi_{n,\ell,p}^{\mathsf{A}}(e_{n,k}) \\ 
    &= \psi_{n,\ell,p}^{\mathsf{A}}(\Omega_{n,p}^{\mathsf{A}}(\hopa{n}{k}{p})). \qedhere
  \end{align*} 
\end{proof}

Theorem~\ref{thm:A} gives rise to a spherical function
$\omega_{n,\ell,\lri}^{\mathsf{A}}$ on $\GL_n(K)$ associated with the
parameters $(1, q, \dots, q^{n-1}, q^{\ell})$. Let
$\varphi^{\mathsf{A}}$ be the characteristic function on
$\mathrm{G}_n^+(K)$ and $\mu^{\mathsf{A}}$ be the normalized Haar
measure on $\GL_n(K)$ such that
$\mu^{\mathsf{A}}(\GL_n(\lri))=1$. In analogy to \eqref{def:Z} we
define the \define{local Ehrhart--Hecke zeta function} of type
$\mathsf{A}$ as
\begin{align}\label{def:Z.local.A}
  \lZA{n}{\ell}{\lri}(s) &= \int_{\GL_n(K)} \varphi^{\mathsf{A}}(g)|\det(g)|_{\mfp}^s \omega^{\mathsf{A}}_{n,\ell,\lri}(g^{-1}) \, \mathrm{d}\mu^{\mathsf{A}}.
\end{align}

If $\lri=\Z_p$, then we get an Ehrhart interpretation of the
$\omega_{n,\ell,p}^{\mathsf{A}}$ similar to that in
Proposition~\ref{prop:spherical}. This implies that, for $P\in \mathscr{P}^{\Lambda_0}$ with $\mathscr{E}_{n,\ell}(P)\neq 0$,
\begin{equation}\label{def:Z.local}
  \lZA{n}{\ell}{p}(s) = \sum_{m = 0}^\infty\EHCA{P}{\ell}(p^m)p^{-ms}.
\end{equation}
We define the \define{global Ehrhart--Hecke zeta function} of type
$\mathsf{A}$ as
\begin{align}\label{def:Z.global.A}
  \gZA{n}{\ell}(s) &= \sum_{m=1}^{\infty} \EHCA{P}{\ell}(m) m^{-s} = \prod_{p\text{ prime}} \lZA{n}{\ell}{p}(s).
\end{align}
Work of Tamagawa yields a simple formula for $\lZA{n}{\ell}{\lri}$,
far less involved than the formulae for $\lZC{n}{\ell}{\lri}$ given in
Theorem~\ref{thmabc:BIgu}. Although Gunnells and Rodriguez Villegas
did not work with the zeta function explicitly, its formula can be
deduced from the discussion in~\cite[Sec.~1.10]{GRV/07}.

\begin{prop}\label{prop:Z.A}
  For all $\ell\in \Z$,
  \begin{align*}
    \lZA{n}{\ell}{\lri}(s) &= (1-q^{\ell-s})^{-1} \prod_{k=1}^{n-1} (1-q^{k-s})^{-1}, & \gZA{n}{\ell}(s) &= \zeta(s-\ell)\prod_{k=1}^{n-1} \zeta(s-k).
  \end{align*}
\end{prop}

\begin{proof}
  Tamagawa~\cite{Tamagawa/63} established the identity for the Hecke series in
  $\mcH_{n,\lri}^{\mathsf{A}}\llbracket X \rrbracket$
  \begin{equation}\label{eqn:tamagawa}
    \sum_{m\geq 0} \hopalta{n}{\lri}(\pi^m) X^m = \left( \sum_{k=0}^n (-1)^k q^{\binom{k}{2}}\hopa{n}{k}{\lri}X^k\right)^{-1} ,
  \end{equation}
  where $\hopalta{n}{\lri}(\pi^m)$ is the characteristic function of the
  $\Gamma_n$-double cosets in $\mathrm{G}_n^+(K)$ with a representative with
  determinant $\pi^m$. Applying $\psi_{n,\ell,\lri}^{\mathsf{A}}
  \Omega_{n,\lri}^{\mathsf{A}}$ to~\eqref{eqn:tamagawa} and setting $X=q^{-s}$
  yields 
  \begin{align*}
    \sum_{m\geq 0} \psi_{n,\ell,\lri}^{\mathsf{A}}(\Omega_{n,\lri}^{\mathsf{A}}(\hopalta{n}{\lri}(\pi^m))) q^{-ms} &= \left( \sum_{k=0}^n (-1)^{k}\psi_{n,\ell, \lri}^{\mathsf{A}}(e_{n,k})q^{-ks}\right)^{-1}\\ 
    &= \left(1-q^{\ell-s}\right)^{-1} \prod_{k=1}^{n-1} \left(1-q^{k-s}\right)^{-1} .
  \end{align*}
  Since $\Phi_{n,k,\ell}^{\mathsf{A}}(q)$ is an eigenvalue for $\hopa{n}{k}{\lri}$
  by Theorem~\ref{thm:A}, the result holds. 
\end{proof}

\section{Multiplicativity: proof of Theorem~\ref{thmabc:mult}}\label{sec:multi}

We treat both the types $\mathsf{A}$ and $\mathsf{C}$ simultaneously:
let $\mathsf{X} \in \{\mathsf{A}, \mathsf{C}\}$. Write
\begin{align}\label{eqn:global-pair}
  \left(\mathrm{G}^+_{\mathsf{X},n},\ \Gamma_{\mathsf{X},n}\right) &= 
  \begin{cases}
    \left(\GL_{n}(\Q) \cap \Mat_{n}(\Z),\ \GL_{n}(\Z)\right) & \text{ if } \mathsf{X} = \mathsf{A}, \\ 
    \left(\GSp_{2n}(\Q) \cap \Mat_{2n}(\Z),\ \GSp_{2n}(\Z)\right) & \text{ if } \mathsf{X} = \mathsf{C}.
  \end{cases} 
\end{align}
The pair in~\eqref{eqn:global-pair} is a Hecke pair, with associated
Hecke ring~$\mcH_{\Z}^{\mathsf{X}}$. For $m\in \N$, let
$T_n^{\mathsf{X}}(m)\in \mcH_{\Z}^{\mathsf{X}}$ be the characteristic
function of the $\Gamma_{\mathsf{X},n}$-double cosets with a
representative of determinant $\pm m$. Let
\begin{align*}
  \mathscr{D}_n^{\mathsf{X}}(m) &= \left\{\Gamma_{\mathsf{X},n}g \in \Gamma_{\mathsf{X},n}\setminus \Gamma_{\mathsf{X},n}\mathrm{G}^+_{\mathsf{X},n}\Gamma_{\mathsf{X},n} ~\middle|~ |\det(g)| = m \right\} . 
\end{align*}
We define an action of $\mcH_{n,p}^{\mathsf{X}}$ on $\UIV_{\tin}$,
analogous to~\eqref{eqn:HeckeC-action}. This yields, in particular,
for $\ell\in [\tin]_0$ and $P\in \mathscr{P}^{\Lambda_0}$, that
\begin{align}\label{eqn:global-Hecke-action}
  \EHCX{P}{\ell}(m) = \dfrac{T_{n}^{\mathsf{X}}(m) \mathscr{E}_{\tin,\ell}(P)}{\mathscr{E}_{\tin,\ell}(P)} &= \dfrac{1}{\mathscr{E}_{\tin,\ell}(P)}\sum_{\Gamma_{\mathsf{X},n}g \in \mathscr{D}_n^{\mathsf{X}}(m)} \mathscr{E}_{\tin,\ell}(g\cdot P) .
\end{align}

\subsection{Proof of Theorem~\ref{thmabc:mult}}

Let $a,b\in \N$ be coprime. By~\cite[Cor.~V.6.2 \&~VI.4.2]{Krieg/90}, 
\begin{align*}
  T_n^{\mathsf{X}}(ab) &= T_n^{\mathsf{X}}(a)T_n^{\mathsf{X}}(b).
\end{align*}
Furthermore, the Hecke ring decomposes into an infinite tensor product:
\begin{align}\label{eqn:tensor-prod}
  \mcH_{\Z}^{\mathsf{X}} &\cong \bigotimes_{p\textup{ prime}} \mcH_{n,p}^{\mathsf{X}}. 
\end{align}
The isomorphism in~\eqref{eqn:tensor-prod} follows from \cite[Thm.~V.6.4
\&~VI.4.3]{Krieg/90} together with the fact that $\mcH_{n,p}^{\mathsf{X}}$ is
isomorphic to the $p$-primary component of $\mcH_{\Z}^{\mathsf{X}}$. In particular,
by Lemma~\ref{lem:non-Arch-to-Arch} the $p$-primary component of
$\mcH_{\Z}^{\mathsf{X}}$ and $\mcH_{n,p}^{\mathsf{X}}$ act on
$\UIV_{\tin}$ in the same way. 

We consider the case $\mathsf{X}=\mathsf{C}$. Suppose $a =
p_{1}^{e_{1}}\cdots p_{r}^{e_{r}}$ and $b = q_{1}^{f_{1}}\cdots
q_{s}^{f_{s}}$ are the prime decompositions for pairwise distinct
primes $p_i$ and $q_j$. For each $i\in [r]$ and $j\in [s]$, there
exist polynomials $\alpha_i,\beta_j\in\Z[\bm{x}]$ such that
\begin{align*}
  T_{n}^{\mathsf{C}}(p_{i}^{e_{i}}) &= \alpha_i\left(\hopc{n}{0}{p_{i}}, \hopc{n}{1}{p_{i}}, \dots, \hopc{n}{n}{p_{i}}\right), \\ 
  T_{n}^{\mathsf{C}}(q_{j}^{f_{j}}) &= \beta_j\left(\hopc{n}{0}{q_{j}}, \hopc{n}{1}{q_{j}}, \dots, \hopc{n}{n}{q_{j}}\right) .
\end{align*}
Since $a$ and $b$ are coprime, we have
\begin{align}\label{eqn:big-prod}
  T_n^{\mathsf{C}}(ab) &= \prod_{i=1}^{r}\alpha_i\left(\hopc{n}{0}{p_{i}}, \dots, \hopc{n}{n}{p_{i}}\right)\prod_{j=1}^{s}\beta_j\left(\hopc{n}{0}{q_{j}}, \dots, \hopc{n}{n}{q_{j}}\right) .
\end{align}
Thus, by Theorem~\ref{thmabc:ehr.sat},~\eqref{eqn:global-Hecke-action},
and~\eqref{eqn:big-prod} we have
\begin{align*}
  \EHCC{P}{\ell}(ab) &= \prod_{i=1}^{r}\alpha_i\left(\Phi_{n,0,\ell}^{\mathsf{C}}(p_{i}), \dots, \Phi_{n,n,\ell}^{\mathsf{C}}(p_{i})\right)\prod_{j=1}^{s}\beta_j\left(\Phi_{n,0,\ell}^{\mathsf{C}}(q_{j}), \dots, \Phi_{n,n,\ell}^{\mathsf{C}}(q_{j})\right) \\
  &= \dfrac{T_n^{\mathsf{C}}(a)\mathscr{E}_{2n,\ell}(P)}{\mathscr{E}_{2n,\ell}(P)} \cdot \dfrac{T_n^{\mathsf{C}}(b)\mathscr{E}_{2n,\ell}(P)}{\mathscr{E}_{2n,\ell}(P)} = \EHCC{P}{\ell}(a)\EHCC{P}{\ell}(b) .
\end{align*}
A similar argument applies to the case $\mathsf{X} = \mathsf{A}$, completing the
proof. \qed

\section{Asymptotics and non-negativity: proof of Theorem~\ref{thmabc:asy}}

Theorem~\ref{thmabc:BIgu} and Proposition~\ref{prop:Z.A} provide
formulae for the local Ehrhart--Hecke zeta
functions~$\lZX{n}{\ell}{\lri}$. In Section~\ref{subsec:analytic} we
leverage these formulae to obtain fundamental analytic properties of
Euler products over these rational functions. In case $\msfC$ for
$\ell\in\{0,2n\}$, the numerator may be expressed as a polynomial with
non-negative coefficients (see Proposition~\ref{prop:l=0.2n}); for
$\ell=n$ we conjecture such an expression (see
Conjecture~\ref{conj:l=n}). In Section~\ref{subsec:asy} we use these
expressions to prove the asymptotic formulae for the partial sums of
the values of the functions~$\EHCX{P}{\ell}$ given in
Theorem~\ref{thmabc:asy}. In Section~\ref{subsec:nonneg} we trace a
non-negativity phenomenon for $\ell\in\{0,n,2n\}$ by linking the
rational functions $\lZC{n}{\ell}{\lri}$ to Igusa functions.

\subsection{Analytic properties}\label{subsec:analytic}

Let $F$ be a number field with ring of integers $\Gri$. Given a non-zero prime
ideal $\mfp$ of $\Gri$, we write $\Gri_{\mfp}$ for the completion of $\Gri$
at~$\mfp$. Write
\begin{equation}\label{def:EP.Z}
  \lZX{n}{\ell}{F}(s) = \prod_{\mfp \in\Spec(\Gri)\setminus\{ 0\}}
  \lZX{n}{\ell}{\Gri_{\mfp}}(s)
\end{equation}
for the \define{global Ehrhart--Hecke series} in type
$\mathsf{X}\in\{\mathsf{A},\mathsf{C}\}$. In the special case $F=\Q$,
we recover $\lZX{n}{\ell}{F}= \gZX{n}{\ell}$ as
in~\eqref{equ:global.Z} and~\eqref{def:Z.global.A}. Note that
$\lZA{n}{\ell}{F}$ is a product of translates of Dedekind zeta
functions; see Proposition~\ref{prop:Z.A}. Its abscissa of convergence
is evidently $\max(n, \ell+1)$ and it has meromorphic continuation to
the whole complex plane. Hence we focus on the type $\mathsf{C}$ case.

\begin{prop}\label{prop:abscissa}
  The global Ehrhart--Hecke zeta function
  $\lZC{n}{\ell}{F}(s)$ has abscissa of convergence
  \begin{align}\label{def:abs.C}
    \alpha_{n,\ell}^{\mathsf{C}} = \frac{n+1}{2} + \frac{1 + \max(0,\ell-n)}{n}
  \end{align}
  and allows for meromorphic continuation to (at least) a complex
  half-plane of the form $\{s\in\C \mid \Re(s) >
  \alpha_{n,\ell}^{\mathsf{C}}-\varepsilon\}$ for some $\varepsilon =
  \varepsilon_{n,\ell}\in\R_{>0}$, and even the whole complex plane
  if~$n\leq 2$. In any case, the continued function has a pole at
  $\alpha_{n,\ell}^{\mathsf{C}}$. For $\ell = n$ it has order two,
  otherwise it is simple.
\end{prop}

Our proof uses Theorem~\ref{thmabc:BIgu} which gives, applied to $\lri
= \Gri_\mfp$, an explicit formula for each of the factors of the Euler
product~\eqref{def:EP.Z}. Recall the definition~\eqref{eqn:beta-def} of $\beta$.

\begin{lem}\label{lem:no-other-constants}
  Let $m\in [2n]_0$, $k\in [m]_0$, $I\in \binom{[n]}{k}$, and $J\in
  \binom{[n]}{m-k}$. If $\ell\leq n$, then
  \begin{align*}
    \beta(I)+\beta(J) + (m-k)(\ell - n)-m\binom{n+1}{2} - m \leq 0.
  \end{align*}
  If $\ell > n$, then
  \begin{align*}
    \beta(I)+\beta(J) - k(\ell - n)-m\binom{n+1}{2} \leq 0.
  \end{align*}
  In both cases, equality holds if and only if $m=0$.
\end{lem}

\begin{proof}
  Consider the function $f:[0,n]\to \R, \, x\mapsto x(2n+1 - x)$, with
  global maximum at $x=n$. For $i\in [n]$, we have
  $f(i)=2\beta(\{i\})$. Thus, $\beta(I) + \beta(J) \leq
  m\binom{n+1}{2}$. This inequality is an equality if either $m=0$ or
  \begin{align}\label{eqn:one-case}
    (m,k,I,J) &= (2,1,\{n\},\{n\}) .
  \end{align}
  
  If $\ell\leq n$, then $0\leq m + (m-k)(n-\ell)$. Moreover
  if~\eqref{eqn:one-case} holds, then $0 < m + (m-k)(n-\ell)$. If $\ell > n$,
  then $0\leq k(\ell-n)$, and it is an equality only when $k=0$. Hence, the
  lemma follows.
\end{proof}

\begin{proof}[Proof of Proposition~\ref{prop:abscissa}]
  Let $\alpha = \alpha_{n,\ell}^{\mathsf{C}}$ be as
  in~\eqref{def:abs.C}. We define
  \begin{align*}
    \mathcal{F}_{n,\ell}(s) &= \dfrac{\lZC{n}{\ell}{F}(s)}{\zeta(ns-n\alpha+1)^{1+\delta_{n,\ell}}} . 
  \end{align*}
  We show that there exists some $\epsilon >0$ such that $\mathcal{F}_{n,\ell}(s)$ is
  holomorphic on the right half-plane determined by $s=\alpha-\epsilon$.
  For a variable $X$, we define (recalling~\eqref{eqn:Theta-def}) 
  \begin{align*}
    N_{n,\ell}(X) &=
    \sum_{m=0}^{2n}\sum_{k=0}^m\sum_{I\in\binom{[n]}{k}}\sum_{J\in\binom{[n]}{m-k}}\Theta_{n,I,J}(X^{-1})
    X^{\beta(I)+\beta(J) + (m-k)(\ell - n) - mn\alpha} .
  \end{align*}
  The Laurent polynomial $N_{n,\ell}(X) \in \Z[X^{-1}]$ has constant
  term $1$ by Lemma~\ref{lem:no-other-constants} and is monic by
  Corollary~\ref{cor:func-eq} and Lemma~\ref{lem:com.den}.  Therefore,
  $N_{n,\ell}(r)\neq 0$ for all $r\in\Q\setminus\{\pm 1\}$. By
  Lemma~\ref{lem:com.den}
  \begin{align*}
    \mathcal{F}_{n,\ell}(\alpha) &= \prod_{\mfp \in\Spec(\Gri)\setminus\{ 0\}} \dfrac{(1 - q_{\mfp}^{-1})^{1+\delta_{n,\ell}}N_{n,\ell}(q_{\mfp})}{\prod_{i\in [n]}(1 -
    q_{\mfp}^{\beta(\{i\})-n\alpha})\prod_{j\in [n]}(1 - q_{\mfp}^{\beta(\{j\}) - n + \ell - n\alpha})} ,
  \end{align*}
  where $q_{\mfp} = \#(\Gri_{\mfp}/\mfp)$. Hence
  $\mathcal{F}_{n,\ell}(s)$ converges, and thus is holomorphic, on the
  right half-plane determined by $s=\alpha-\epsilon$ for some
  $\epsilon>0$.
\end{proof}

\subsection{Asymptotics and proof of Theorem~\ref{thmabc:asy}}\label{subsec:asy}
The asymptotics in type $\mathsf{A}$ are simpler as $\gZA{n}{\ell}$ is
a product of Riemann zeta functions. The following refines
Theorem~\ref{thmabc:asy} in type~$\msfA$.

\begin{prop}\label{prop:type-A-asymptotics}
  Let $\ell\in [n]_0$ and $P\in\mathscr{P}^{\Lambda_0}$ with
  $\mathscr{E}_{n,\ell}(P) \neq 0$. As $N\to \infty$,
  \begin{align*}
    \sum_{m=1}^N \EHCA{P}{\ell}(m) &~\sim~ \begin{cases}
      \frac{\zeta(n-\ell)}{n}\zeta(2) \zeta(3)\cdots \zeta(n-1) N^n & \text{ if }\ell\leq n-2, \\[0.25em]
      \frac{1}{n}\zeta(2) \zeta(3)\cdots \zeta(n-2) N^n\log N & \text{ if }\ell = n-1, \\[0.25em]
      \frac{1}{n+1}\zeta(2) \zeta(3)\cdots \zeta(n) N^{n+1} & \text{ if }\ell = n .
    \end{cases}
  \end{align*}
\end{prop}

\begin{proof}
Immediate from Proposition~\ref{prop:Z.A}
\end{proof}

\begin{proof}[Proof of Theorem~\ref{thmabc:asy}]
  For $\mathsf{X}=\mathsf{A}$, apply
  Proposition~\ref{prop:type-A-asymptotics}, so we may
  assume~$\mathsf{X}=\mathsf{C}$.  The result follows from
  Theorem~\ref{thmabc:BIgu} in the special case $\lri = \Zp$ and a
  standard application of the Tauberian theorem; see, for instance,
  \cite[Thm.~4.20]{dSG/00}. The meromorphic continuability to the left
  of the line determined by $\alpha_{n,\ell}^{\mathsf{C}}$ established
  in Proposition~\ref{prop:abscissa} implies that the latter is
  applicable. The $\log$-term in the case $\ell = n$ reflects the
  duplicity of the pole at $s = \alpha_{n,n}^{\mathsf{C}}$.
\end{proof}

\begin{remark}\label{rem:constants}
  The constants $\consymC{n}{\ell}$ in Theorem~\ref{thmabc:asy} are equal to
  \begin{align*}
    \dfrac{2n k_{n,\ell}}{(1 + \delta_{n,\ell})(n^2 + \max(n, 2\ell-n) + 2)},
  \end{align*}
  where 
  \begin{align*}
    k_{n,\ell} &= \lim_{s\to \alpha_{n,\ell}^{\mathsf{C}}} \dfrac{\gZC{n}{\ell}(s)}{\zeta(ns-n\alpha_{n,\ell}^{\mathsf{C}} + 1)^{1 + \delta_{n,\ell}}} . 
  \end{align*}
  In the case $n=2$, we conclude by~\eqref{eqn:Z2ell} that 
  \begin{align*}
    \consymC{2}{0} &= \frac{7\zeta(3)}{8}, &
    \consymC{2}{1} &= \frac{\zeta(2)\zeta(3)}{12\zeta(5)}, &
    \consymC{2}{2} &= \frac{5}{8}, &
    \consymC{2}{3} &= \frac{\zeta(2)\zeta(3)}{15\zeta(5)}, &
    \consymC{2}{4} &=  \frac{7\zeta(3)}{12}.
  \end{align*}
  For general $n$ and $\ell\in[2n]_0$, the constants $\consymC{n}{\ell}$ may, in
  principle, be expressed as Euler products involving the complicated polynomial
  expressions in $p^{-1}$ given by the numerators of the rational functions
  $\lZC{n}{\ell}{p}$ in Lemma~\ref{lem:com.den}. Indeed, the proof of
  Proposition~\ref{prop:abscissa} shows that the abscissa of convergence arises
  from the Euler product over the denominators of these rational functions,
  whereas the product over their numerators converges strictly better than the
  denominator product. For the special values $\ell\in\{0,n,2n\}$, we give more
  details about $\consymC{n}{\ell}$ in Section~\ref{subsec:nonneg}.
\end{remark}

\subsection{Non-negativity}\label{subsec:nonneg}

For general $n$ and $\ell\in[2n]_0$, the numerators of the rational
functions $\lZC{n}{\ell}{\lri}$ given in Lemma~\ref{lem:com.den} have
both negative and positive coefficients in $q$ and $t=q^{-s}$; already
the examples for $n\in\{2,3\}$ in Section~\ref{subsec:EH.exa} show
this. For $\ell\in\{0,2n\}$, however, we prove that curious
cancellations yield expressions featuring \emph{non-negative}
coefficients; see Proposition~\ref{prop:l=0.2n}. For $\ell=n$ we
conjecture such a formula; see Conjecture~\ref{conj:l=n}.
Computational evidence supports the supposition that $0$, $n$, and
$2n$ are the only special $\ell$-values for which such cancellations
occur.

To this end we relate the functions $\lZC{n}{\ell}{\lri}$ for these
$\ell$ with Igusa functions. To formulate this connection, we denote
the \define{descent set} of a permutation $w\in S_n$ by
\begin{align*}
  \Des(w) &= \left\{i\in [n-1] ~\middle|~ w(i) > w(i + 1) \right\} . 
\end{align*}
Three well-studied permutation statistics are relevant for us: the
\define{descent number}, \define{inversion number}, and \define{major
  index}. They are defined as follows:
\begin{align*}
  \des(w) &= \# \Des(w), \\
  \inv(w) &= \#\{(i,j) ~|~ 1\leq i < j \leq n,\ w(i) > w(j)\}, \\
  \mathrm{maj}(w) &= \sum_{i\in \Des(w)} i .
\end{align*}
Recall, e.g.\ from \cite[(1.3)]{MV/24}, the notion of the \define{Igusa
function} of degree $n$
\[ 
  \Ig_n(Y;X_1,\dots,X_n) = \sum_{I\subseteq[n]} \binom{n}{I}_Y~
  \prod_{i\in I} \frac{X_i}{1-X_i} = \frac{\sum_{w\in S_n}
  Y^{\inv(w)}\prod_{i\in \Des(w)} X_i}{\prod_{i=1}^n (1-X_i)}.
\]
The numerator of the right-hand side is a polynomial with non-negative
coefficients.

\begin{prop}[$\ell = 0$]\label{prop:l=0.2n} We have
  \begin{align*}
    \lZC{n}{0}{\lri}(s) &= \frac{1}{1-q^{\binom{n+1}{2} - ns}}
    \Ig_n\left(q^{-1}; \left(q^{\binom{n+1}{2} - \binom{i+1}{2} -
      ns}\right)_{i\in[n]}\right),
  \end{align*}
\end{prop}

\begin{proof}
   As $\mathscr{E}_{2n,0} = 1$ we get, by \cite[Lem.~9.3]{MV/24} with
   $X_i = q^{\binom{n+1}{2} - \binom{i+1}{2} - ns}$ for $i\in[n]$:
  \begin{align*} 
    \lZC{n}{0}{\lri}(s) &= \int_{\mathrm{G}^+_n(K)}|\det (g)|_{\mfp}^{s}\, \tud \mu = \frac{\sum_{w\in S_n} q^{-\inv(w)} \prod_{i\in \Des(w)}{X_i}}{\prod_{i=0}^n\left(1-{X_i}\right)}. \qedhere
  \end{align*}
\end{proof}

Corollary~\ref{corabc:refl} implies, of course, a similar formula for
$\lZC{n}{2n}{\lri}$. The examples in Section~\ref{subsec:EH.exa} show
that expressing the rational functions $\lZC{n}{\ell}{\lri}$ as
quotients of coprime polynomials does not yield numerator
non-negativity for generic $\ell\in[2n-1]$. Ever so remarkably, the
formulae for $\ell=n$ (for which Corollary~\ref{corabc:refl} is
vacuous) \emph{do} seem to fit the pattern of
Proposition~\ref{prop:l=0.2n}.  Computations for $n\leq 13$ support
the following.

\begin{conj}[$\ell = n$]\label{conj:l=n}
  We have
  \begin{align*}
    \lZC{n}{n}{\lri}(s) &= \frac{1}{1-q^{\binom{n+1}{2}-ns}} \Ig_n\left(q^{-1}; \left( q^{\binom{n+1}{2} - \binom{i}{2}-ns}\right)_{i\in[n]}\right).
 \end{align*}
\end{conj}

Proposition~\ref{prop:l=0.2n} and Conjecture~\ref{conj:l=n} give us formulae for
the numbers $\consymC{n}{\ell}$ from Theorem~\ref{thmabc:asy} in terms of
permutation statistics for $\ell\in\{0,n,2n\}$. For $w\in S_n$, define 
\begin{align*}
  \mathrm{binv}(w) &= \inv(w) + \sum_{i\in\Des(w)}\binom{i+1}{2}.
\end{align*}
We define real numbers 
\begin{align*}
  \gamma_{n,0} &= \prod_{p\textup{ prime}}\left(\sum_{w\in S_n} p^{-(\binv+\des)(w)} \right), \\
  \gamma_{n,n} &= \prod_{p\textup{ prime}}\left(\sum_{w\in S_n} p^{-(\binv+\des-\maj)(w)} \right) . 
\end{align*}
Bright and Savage~\cite[Thm.~4.1]{BS/10} give a formula for the distribution of
$(\binv,\maj)$, and MacMahon gives one for $(\des,\maj)$;
see~\cite[Eq.~(I.4)]{GG/79}. Garsia and Gessel \cite[Sec.~1]{GG/79} give a
formula for the distribution of $(\des,\inv,\maj)$. Much less seems to be known
about the pair $(\binv,\des)$ and the triple~$(\binv,\des,\maj)$.

\begin{table}[h]
 \begin{tabular}{r|c}
    $n$ & $p$th factor of $\gamma_{n,0}$ at $Y=p^{-1}$ \\ \hline \\[-1em]
    $2$ & $1 + Y^3$ \\[0.1em]
    $3$ & $1 + Y^3 + Y^4 + Y^5 + Y^6 + Y^9$ \\[0.1em]
    $4$ & $1 + Y^3 + Y^4 + 2Y^5 + 2Y^6 + Y^7 + 2Y^8 + 2Y^9 + \cdots + Y^{19}$
  \end{tabular}

  \vspace{1em}

  \begin{tabular}{r|c}
    $n$ & $p$th factor of $\gamma_{n,n}$ at $Y=p^{-1}$ \\ \hline \\[-1em]
    $2$ & $1 + Y^2$ \\[0.1em]
    $3$ & $1 + Y^2 + 2Y^3 + Y^4 + Y^6$ \\[0.1em]
    $4$ & $1 + Y^2 + 2Y^3 + 3Y^4 + 2Y^5 + 3Y^6 + \cdots + Y^{13}$
  \end{tabular}
  \caption{Euler factors of $\gamma_{n,0}$ and $\gamma_{n,n}$ at
    $Y=p^{-1}$ for some $n$.}
  \label{tab:gamma}
\end{table}

In Table~\ref{tab:gamma} we collect the Euler factors of the numbers
$\gamma_{n,0}$ and $\gamma_{n,n}$ for $n\in\{2,3,4\}$. Note that
Corollary~\ref{cor:func-eq} implies that these polynomials are
palindromic, whence the information given in the table is complete for
these values of~$n$.

\begin{prop}\label{cor:asymp_l=n}
  We have
  \begin{align*}
    \consymC{n}{0} &= \gamma_{n,0} \cdot \dfrac{2n\prod_{i=1}^{n} \zeta\left(\binom{i+1}{2} + 1\right)}{n^2 + n + 2} , & \consymC{n}{2n} &= \gamma_{n,0} \cdot \dfrac{2n\prod_{i=1}^{n} \zeta\left(\binom{i+1}{2} + 1\right)}{n^2 + 3n + 2} .
  \end{align*}
  Assuming Conjecture~\ref{conj:l=n}, 
  \begin{align*}
    \consymC{n}{n} =  \gamma_{n,n} \cdot \dfrac{ n\prod_{i=1}^{n-1} \zeta\left(\binom{i+1}{2} + 1\right)}{n^2 + n + 2}.
  \end{align*}
\end{prop}

\begin{proof}
  Apply \cite[Thm.~4.20]{dSG/00}, Proposition~\ref{prop:l=0.2n}, and
  Corollary~\ref{corabc:refl}; see Remark~\ref{rem:constants}.
\end{proof}

\section{Examples}\label{sec:exa}

In Section~\ref{subsec:hecke.exa} we exemplify some of the objects and
results of this paper in the case $n=2$ in type $\msfA$. (It
coincides, of course, with the case $n=1$ in type~$\msfC$.) In
Section~\ref{subsec:EH.exa} we give examples of the local
Ehrhart--Hecke zeta functions~$\lZC{n}{\ell}{\lri}$.

\subsection{Ehrhart coefficients as functions on affine buildings}\label{subsec:hecke.exa}

As before let $n\in\N$, $\ell\in[n]_0$, $P$ a lattice polytope, and
$\Lambda\supseteq \Z^n = \Lambda_0$ a lattice. We first note that, for
all $m\in\N$, we have $\mathscr{E}_{n,\ell}^{m^{-1}\Lambda}(P) =
m^{\ell}\mathscr{E}_{n,\ell}^{\Lambda}(P)$, so homothetic lattices
yield ``homothetic'' Ehrhart coefficients. We further observe that the
rational homothety class $[\Lambda]$ of $\Lambda$ contains a unique
\define{$m$-primitive} lattice $\Lambda_{\mathrm{min}} =
\Lambda_{\mathrm{min},m}$, characterized by the fact that
$\Lambda_0\subseteq \Lambda_{\mathrm{min}}$ and
$\Lambda_0\not\subseteq m\Lambda_{\mathrm{min}}$. Thus we focus on
$m$-primitive lattices. In the case that $m$ is a prime, the affine
Bruhat--Tits building associated with the group $\SL_{n}(\Qp)$ offers
a coordinate-free model for the set of $p$-primitive lattices in
$\Q^n$, viewed as the set of vertices of a simplicial complex; see
\cite[Sec.~19]{Garrett}.

We now focus on the case $n=2$, assume that $P$ is a lattice polygon,
and consider $\ell=1$. A well-known result (``Pick's Theorem''; see
\cite[Ex.~1.25]{BS/18}) expresses the linear coefficient
$c_{1}(E^{\Lambda}_{P})$ of the quadratic Ehrhart polynomial
$E^{\Lambda}_P$ of $P$ with respect to $\Lambda$ in terms of the
number of points on the boundary of $P$ contained in~$\Lambda$.

We further chose $m=p=2$, so the relevant building is the $3$-regular
\emph{Bruhat--Tits tree}. We consider the functions
$\mathscr{E}^\Lambda_{2,1}$, for various $2$-power co-index
superlattices $\Lambda\supseteq\Z^2$. In this way we obtain a
$\Q$-valued function $\eta_{P}$ on the vertices of the Bruhat--Tits
tree qua $\eta_{P}([\Lambda]) =
\mathscr{E}_{2,1}^{\Lambda_{\mathrm{min},2}}(P)$. Figure~\ref{fig:tree}
gives some values of two such functions $\eta_{P}$ and $\eta_{P'}$ for
two polygons $P$ and $P'$. The central vertices are associated with
the lattice $\Lambda_0 = \Z^2$.

Evidently $\eta_P$ depends significantly on the polygon~$P$. To see
the independence of the Hecke eigenvalues established in
Theorem~\ref{thm:A} (or, equivalently, Theorem~\ref{thmabc:ehr.sat}),
we add the functions' values in concentric circles around
$[\Lambda_0]$ and divide by the ``central'' values.  For the
concentric circle of radius $2$, these normalized sums are
\[
  \frac{3+4+3+4+4+7}{5/2} = 10 = \frac{5+5+5+5+5+5}{3}.
\]
Invisible in the figure is the lattice $2^{-1}\Lambda_0 \supsetneq
\Lambda_0$. This contributes $2$ to the normalized sum, bringing the
total to $12$, which is the $p^{-2s}$-term of $\lZA{2}{1}{p}(s) =
\lZC{1}{1}{p}(s) = 1 + 2p\cdot p^{-s} + 3p^2 \cdot p^{-2s} + \cdots$
when $p=2$. Furthermore, $3p^2 = \Phi_{1,1,1}^{\mathsf{A}}(p)^2 -
\Phi_{1,2,1}^{\mathsf{A}}(p)$, reflecting that $\hopalta{2}{p}(p^2) =
(\hopa{2}{1}{p})^2 - \hopa{2}{2}{p}$ in $\mcH_{2,p}^{\mathsf{A}}$.

\begin{figure}[h]
  \centering
    \begin{tikzpicture}
      \pgfmathsetmacro{\width}{3}
      \pgfmathsetmacro{\height}{2.5}
      \pgfmathsetmacro{\colwidth}{2.9}
      \pgfmathsetmacro{\colheight}{7.5}

      \draw[black,rounded corners=2pt] 
        (-1.95*\width, -\height-3.4) rectangle (-\width+\colwidth, \height+1.45);

      \draw[black,rounded corners=2pt] 
        (\width-\colwidth, -\height-3.4) rectangle (1.95*\width, \height+1.45);

      \node at (-\width,\height) {
        \begin{tikzpicture}
          \coordinate (A) at (0,0);
          \coordinate (B) at (1,0);
          \coordinate (C) at (0,1);
          \coordinate (D) at (2,1);
        
          \draw[thick, fill=blue!50] (A) -- (B) -- (D) -- (C) -- cycle;

          \foreach \point in {A, B, C, D}
            \fill[black] (\point) circle (2pt);

          \node[below left] at (A) {$(0,0)$};
          \node[below right] at (B) {$(1,0)$};
          \node[above left] at (C) {$(0,1)$};
          \node[above right] at (D) {$(2,1)$};
        \end{tikzpicture}
      };
      \node at (-\width,-\height) {
        \begin{tikzpicture}
          \node[circle, draw, fill=blue!20, inner sep=0.3pt] (0) at (0,0) {{\scriptsize $5/2$}};
          \node[circle, draw, fill=blue!20, inner sep=1pt] (a) at (120:1) {$3$};
          \node[circle, draw, fill=blue!20, inner sep=1pt] (b) at (240:1) {$4$};
          \node[circle, draw, fill=blue!20, inner sep=1pt] (c) at (360:1) {$3$};
          \node[circle, draw, fill=blue!20, inner sep=1pt] (a1) at (90:1.75) {$4$};
          \node[circle, draw, fill=blue!20, inner sep=1pt] (a2) at (150:1.75) {$3$};
          \node[circle, draw, fill=blue!20, inner sep=1pt] (b1) at (210:1.75) {$7$};
          \node[circle, draw, fill=blue!20, inner sep=1pt] (b2) at (270:1.75) {$4$};
          \node[circle, draw, fill=blue!20, inner sep=1pt] (c1) at (330:1.75) {$4$};
          \node[circle, draw, fill=blue!20, inner sep=1pt] (c2) at (390:1.75) {$3$};
          \node[circle, draw, fill=blue!20, inner sep=1pt] (a1i) at (75:2.25) {$6$};
          \node[circle, draw, fill=blue!20, inner sep=1pt] (a1ii) at (105:2.25) {$4$};
          \node[circle, draw, fill=blue!20, inner sep=1pt] (a2i) at (135:2.25) {$3$};
          \node[circle, draw, fill=blue!20, inner sep=1pt] (a2ii) at (165:2.25) {$3$};
          \node[circle, draw, fill=blue!20, inner sep=0.5pt] (b1i) at (195:2.25) {{\footnotesize $13$}};
          \node[circle, draw, fill=blue!20, inner sep=1pt] (b1ii) at (225:2.25) {$7$};
          \node[circle, draw, fill=blue!20, inner sep=1pt] (b2i) at (255:2.25) {$4$};
          \node[circle, draw, fill=blue!20, inner sep=1pt] (b2ii) at (285:2.25) {$4$};
          \node[circle, draw, fill=blue!20, inner sep=1pt] (c1i) at (315:2.25) {$6$};
          \node[circle, draw, fill=blue!20, inner sep=1pt] (c1ii) at (345:2.25) {$4$};
          \node[circle, draw, fill=blue!20, inner sep=1pt] (c2i) at (375:2.25) {$3$};
          \node[circle, draw, fill=blue!20, inner sep=1pt] (c2ii) at (405:2.25) {$3$};
          \node (a1ir) at (70:2.75) {};
          \node (a1il) at (80:2.75) {};
          \node (a1iir) at (100:2.75) {};
          \node (a1iil) at (110:2.75) {};
          \node (a2ir) at (130:2.75) {};
          \node (a2il) at (140:2.75) {};
          \node (a2iir) at (160:2.75) {};
          \node (a2iil) at (170:2.75) {};
          \node (b1ir) at (190:2.75) {};
          \node (b1il) at (200:2.75) {};
          \node (b1iir) at (220:2.75) {};
          \node (b1iil) at (230:2.75) {};
          \node (b2ir) at (250:2.75) {};
          \node (b2il) at (260:2.75) {};
          \node (b2iir) at (280:2.75) {};
          \node (b2iil) at (290:2.75) {};
          \node (c1ir) at (310:2.75) {};
          \node (c1il) at (320:2.75) {};
          \node (c1iir) at (340:2.75) {};
          \node (c1iil) at (350:2.75) {};
          \node (c2ir) at (370:2.75) {};
          \node (c2il) at (380:2.75) {};
          \node (c2iir) at (400:2.75) {};
          \node (c2iil) at (410:2.75) {};
  
          \draw[thick] (0) -- (a);
          \draw[thick] (0) -- (b);
          \draw[thick] (0) -- (c);
          \draw[thick] (a) -- (a1);
          \draw[thick] (a) -- (a2);
          \draw[thick] (b) -- (b1);
          \draw[thick] (b) -- (b2);
          \draw[thick] (c) -- (c1);
          \draw[thick] (c) -- (c2);
          \draw[thick] (a1) -- (a1i);
          \draw[thick] (a1) -- (a1ii);
          \draw[thick] (a2) -- (a2i);
          \draw[thick] (a2) -- (a2ii);
          \draw[thick] (b1) -- (b1i);
          \draw[thick] (b1) -- (b1ii);
          \draw[thick] (b2) -- (b2i);
          \draw[thick] (b2) -- (b2ii);
          \draw[thick] (c1) -- (c1i);
          \draw[thick] (c1) -- (c1ii);
          \draw[thick] (c2) -- (c2i);
          \draw[thick] (c2) -- (c2ii);
          \draw[thick] (a1i) -- (a1ir);
          \draw[thick] (a1i) -- (a1il);
          \draw[thick] (a1ii) -- (a1iir);
          \draw[thick] (a1ii) -- (a1iil);
          \draw[thick] (a2i) -- (a2ir);
          \draw[thick] (a2i) -- (a2il);
          \draw[thick] (a2ii) -- (a2iir);
          \draw[thick] (a2ii) -- (a2iil);
          \draw[thick] (b1i) -- (b1ir);
          \draw[thick] (b1i) -- (b1il);
          \draw[thick] (b1ii) -- (b1iir);
          \draw[thick] (b1ii) -- (b1iil);
          \draw[thick] (b2i) -- (b2ir);
          \draw[thick] (b2i) -- (b2il);
          \draw[thick] (b2ii) -- (b2iir);
          \draw[thick] (b2ii) -- (b2iil);
          \draw[thick] (c1i) -- (c1ir);
          \draw[thick] (c1i) -- (c1il);
          \draw[thick] (c1ii) -- (c1iir);
          \draw[thick] (c1ii) -- (c1iil);
          \draw[thick] (c2i) -- (c2ir);
          \draw[thick] (c2i) -- (c2il);
          \draw[thick] (c2ii) -- (c2iir);
          \draw[thick] (c2ii) -- (c2iil);
        \end{tikzpicture}
      };
      \node at (-\width,\height-1.66) {$P$};
      \node at (\width,\height-1.66) {$P'$};
      \node at (\width,\height) {
        \begin{tikzpicture}
          \coordinate (A) at (0,0);
          \coordinate (B) at (1/2,0);
          \coordinate (C) at (0,1/2);
          \coordinate (D) at (1/2,1);
          \coordinate (E) at (3/2,3/2);
          \coordinate (F) at (2,1/2);
        
          \draw[thick, fill=red!50] (A) -- (B) -- (F) -- (E) -- (D) -- (C) -- cycle;

          \foreach \point in {A, B, C, D, E, F}
            \fill[black] (\point) circle (2pt);
        
          \node[below left] at (A) {$(0,0)$};
          \node[below right] at (B) {$(1,0)$};
          \node[above left] at (C) {$(0,1)$};
          \node[above left] at (D) {$(1,2)$};
          \node[above] at (E) {$(3,3)$};
          \node[right] at (F) {$(4,1)$};
        \end{tikzpicture}
      };
      \node at (\width,-\height) {
        \begin{tikzpicture}
          \node[circle, draw, fill=red!20, inner sep=0.8pt] (0) at (0,0) {$3$};
          \node[circle, draw, fill=red!20, inner sep=1pt] (a) at (120:1) {$4$};
          \node[circle, draw, fill=red!20, inner sep=1pt] (b) at (240:1) {$4$};
          \node[circle, draw, fill=red!20, inner sep=1pt] (c) at (360:1) {$4$};
          \node[circle, draw, fill=red!20, inner sep=1pt] (a1) at (90:1.75) {$5$};
          \node[circle, draw, fill=red!20, inner sep=1pt] (a2) at (150:1.75) {$5$};
          \node[circle, draw, fill=red!20, inner sep=1pt] (b1) at (210:1.75) {$5$};
          \node[circle, draw, fill=red!20, inner sep=1pt] (b2) at (270:1.75) {$5$};
          \node[circle, draw, fill=red!20, inner sep=1pt] (c1) at (330:1.75) {$5$};
          \node[circle, draw, fill=red!20, inner sep=1pt] (c2) at (390:1.75) {$5$};
          \node[circle, draw, fill=red!20, inner sep=1pt] (a1i) at (75:2.25) {$7$};
          \node[circle, draw, fill=red!20, inner sep=1pt] (a1ii) at (105:2.25) {$5$};
          \node[circle, draw, fill=red!20, inner sep=1pt] (a2i) at (135:2.25) {$7$};
          \node[circle, draw, fill=red!20, inner sep=1pt] (a2ii) at (165:2.25) {$5$};
          \node[circle, draw, fill=red!20, inner sep=1pt] (b1i) at (195:2.25) {$7$};
          \node[circle, draw, fill=red!20, inner sep=1pt] (b1ii) at (225:2.25) {$5$};
          \node[circle, draw, fill=red!20, inner sep=1pt] (b2i) at (255:2.25) {$5$};
          \node[circle, draw, fill=red!20, inner sep=1pt] (b2ii) at (285:2.25) {$7$};
          \node[circle, draw, fill=red!20, inner sep=1pt] (c1i) at (315:2.25) {$7$};
          \node[circle, draw, fill=red!20, inner sep=1pt] (c1ii) at (345:2.25) {$5$};
          \node[circle, draw, fill=red!20, inner sep=1pt] (c2i) at (375:2.25) {$7$};
          \node[circle, draw, fill=red!20, inner sep=1pt] (c2ii) at (405:2.25) {$5$};
          \node (a1ir) at (70:2.75) {};
          \node (a1il) at (80:2.75) {};
          \node (a1iir) at (100:2.75) {};
          \node (a1iil) at (110:2.75) {};
          \node (a2ir) at (130:2.75) {};
          \node (a2il) at (140:2.75) {};
          \node (a2iir) at (160:2.75) {};
          \node (a2iil) at (170:2.75) {};
          \node (b1ir) at (190:2.75) {};
          \node (b1il) at (200:2.75) {};
          \node (b1iir) at (220:2.75) {};
          \node (b1iil) at (230:2.75) {};
          \node (b2ir) at (250:2.75) {};
          \node (b2il) at (260:2.75) {};
          \node (b2iir) at (280:2.75) {};
          \node (b2iil) at (290:2.75) {};
          \node (c1ir) at (310:2.75) {};
          \node (c1il) at (320:2.75) {};
          \node (c1iir) at (340:2.75) {};
          \node (c1iil) at (350:2.75) {};
          \node (c2ir) at (370:2.75) {};
          \node (c2il) at (380:2.75) {};
          \node (c2iir) at (400:2.75) {};
          \node (c2iil) at (410:2.75) {};
  
          \draw[thick] (0) -- (a);
          \draw[thick] (0) -- (b);
          \draw[thick] (0) -- (c);
          \draw[thick] (a) -- (a1);
          \draw[thick] (a) -- (a2);
          \draw[thick] (b) -- (b1);
          \draw[thick] (b) -- (b2);
          \draw[thick] (c) -- (c1);
          \draw[thick] (c) -- (c2);
          \draw[thick] (a1) -- (a1i);
          \draw[thick] (a1) -- (a1ii);
          \draw[thick] (a2) -- (a2i);
          \draw[thick] (a2) -- (a2ii);
          \draw[thick] (b1) -- (b1i);
          \draw[thick] (b1) -- (b1ii);
          \draw[thick] (b2) -- (b2i);
          \draw[thick] (b2) -- (b2ii);
          \draw[thick] (c1) -- (c1i);
          \draw[thick] (c1) -- (c1ii);
          \draw[thick] (c2) -- (c2i);
          \draw[thick] (c2) -- (c2ii);
          \draw[thick] (a1i) -- (a1ir);
          \draw[thick] (a1i) -- (a1il);
          \draw[thick] (a1ii) -- (a1iir);
          \draw[thick] (a1ii) -- (a1iil);
          \draw[thick] (a2i) -- (a2ir);
          \draw[thick] (a2i) -- (a2il);
          \draw[thick] (a2ii) -- (a2iir);
          \draw[thick] (a2ii) -- (a2iil);
          \draw[thick] (b1i) -- (b1ir);
          \draw[thick] (b1i) -- (b1il);
          \draw[thick] (b1ii) -- (b1iir);
          \draw[thick] (b1ii) -- (b1iil);
          \draw[thick] (b2i) -- (b2ir);
          \draw[thick] (b2i) -- (b2il);
          \draw[thick] (b2ii) -- (b2iir);
          \draw[thick] (b2ii) -- (b2iil);
          \draw[thick] (c1i) -- (c1ir);
          \draw[thick] (c1i) -- (c1il);
          \draw[thick] (c1ii) -- (c1iir);
          \draw[thick] (c1ii) -- (c1iil);
          \draw[thick] (c2i) -- (c2ir);
          \draw[thick] (c2i) -- (c2il);
          \draw[thick] (c2ii) -- (c2iir);
          \draw[thick] (c2ii) -- (c2iil);
        \end{tikzpicture}
      };
      \node at (-\width,-\height-3) {$\eta_{P}$};
      \node at (\width,-\height-3) {$\eta_{P'}$};
    \end{tikzpicture}
  \caption{The functions $\eta_P$ and $\eta_{P'}$ on a Bruhat--Tits
    tree}
  \label{fig:tree}
\end{figure}
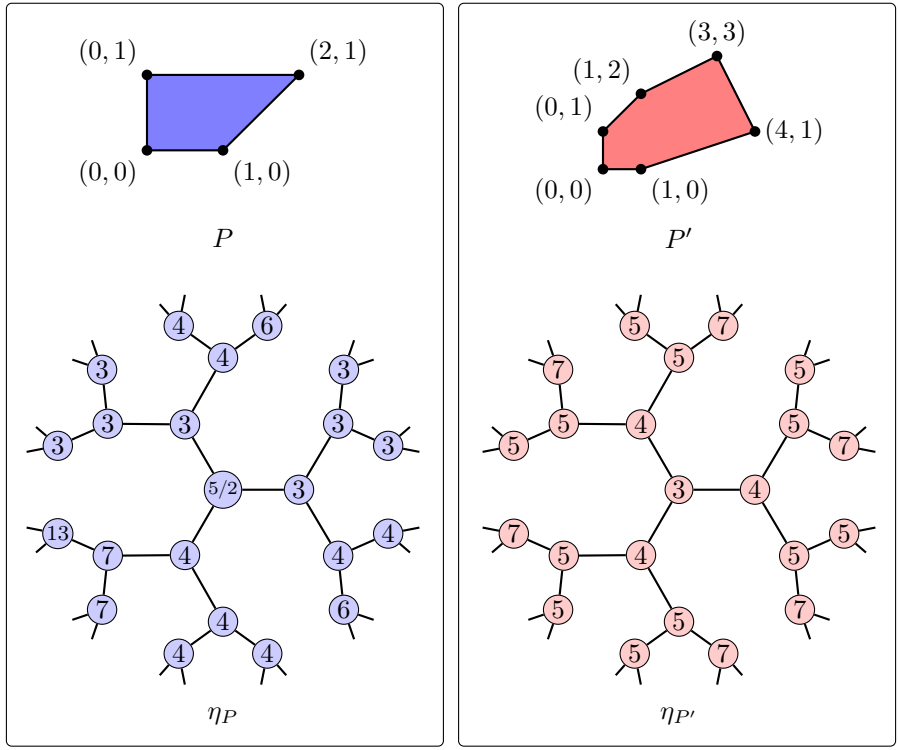

\subsection{Local Ehrhart--Hecke zeta functions}\label{subsec:EH.exa}

We used SageMath~\cite{sagemath} to compute the following examples and
the BRational package~\cite{BRational} to provide nice
formatting. Additional data is freely available on the Zenodo
repository:
\begin{center}
  \url{https://doi.org/10.5281/zenodo.15863412}
\end{center}

For $n\leq 3$ we record $\lZC{n}{\ell}{\lri}$ for all cDVR $\lri$ and
$\ell\in\Z$. Setting $t=q^{-s}$, we have
\begin{align*}
  \lZC{1}{\ell}{\lri}(s) &= \dfrac{1}{(1 - qt) (1 - q^{\ell}t)},\\
  \lZC{2}{\ell}{\lri}(s) &= \dfrac{1 - q^{2+\ell}t^4}{(1 - q^2t^2)(1 - q^3t^2)(1 - q^{\ell} t^2)(1 - q^{1 + \ell} t^2)},\\
  \lZC{3}{\ell}{\lri}(s) &= \dfrac{1 + (q^4 + q^{1+\ell})t^3 - A_{\ell}(q) t^6 + (q^{9+\ell} + q^{6 + 2\ell}) t^9 + q^{10 + 2\ell}t^{12}}{(1 - q^3t^3)(1 - q^5t^3)(1 - q^6t^3)(1 - q^{\ell} t^3)(1 - q^{2 + \ell} t^3)(1 - q^{3 + \ell} t^3)}
  ,
\end{align*}
where $A_{\ell}(q) = q^{3+\ell} + 2q^{4+\ell} + 2q^{6+\ell} + q^{7+\ell}$.

To illustrate the non-negativity phenomena discussed in
Section~\ref{subsec:nonneg}, we write out these functions for
$\ell\in[n]_0$. Corollary~\ref{corabc:refl} yields similar ones for
$\ell\in[n+1,2n]$. These formulae thus underline the special roles
played by $\ell \in \{0,n,2n\}$; see
Section~\ref{subsec:nonneg}. Specifically we have, for $n=2$ {\small
  \begin{align*}
    \lZC{2}{0}{\lri}(s) &= \dfrac{1 + qt^2}{(1 - t^2)(1 - q^2t^2)(1 - q^3t^2)} , \\
    \lZC{2}{1}{\lri}(s) &= \dfrac{1 - q^3t^4}{(1 - qt^2)(1 - q^2t^2)^2(1 - q^3t^2)}, \\
    \lZC{2}{2}{\lri}(s) &= \dfrac{1 + q^2t^2}{(1 - q^2t^2)(1 - q^3t^2)^2}.
   \end{align*}
  }
  \begin{align*}
    \lZC{3}{0}{\lri}(s) &= \dfrac{1 + qt^3 + q^2t^3 + q^3t^3 + q^4t^3
      + q^5t^6}{(1 - t^3)(1 - q^3t^3)(1 - q^5t^3)(1 - q^6t^3)} , \\
       \lZC{3}{1}{\lri}(s) &= \dfrac{1 + q^2t^3 + q^4t^3 - q^4t^6 - 2q^5t^6 - 2q^7t^6 - q^8t^6 + q^8t^9 + q^{10}t^9 + q^{12}t^{12}}{(1 - qt^3)(1 - q^3t^3)^2(1 - q^4t^3)(1 - q^5t^3)(1 - q^6t^3)}, \\
    \lZC{3}{2}{\lri}(s) &= \dfrac{1 + q^3t^3 + q^4t^3 - q^5t^6 - 2q^6t^6 - 2q^8t^6 - q^9t^6 + q^{10}t^9 + q^{11}t^9 + q^{14}t^{12}}{(1 - q^2t^3)(1 - q^3t^3)(1 - q^4t^3)(1 - q^5t^3)^2(1 - q^6t^3)} , \\
    \lZC{3}{3}{\lri}(s) &= \dfrac{1 + q^3t^3 + 2q^4t^3 + q^5t^3 + q^8t^6}{(1 - q^3t^3)(1 - q^5t^3)(1 - q^6t^3)^2}.
  \end{align*}

\begin{acknowledgements}
 We thank Raman Sanyal for inspiring discussions on the subject of
 this paper, in particular the connections with \cite{BK/85}. We are
 grateful to the Bielefeld Center for Interdisciplinary Research (ZiF)
 for providing the backdrop for a workshop in 2021. Alfes and Voll
 were funded by the Deutsche Forschungsgemeinschaft (DFG, German
 Research Foundation) -- Project-ID 491392403 -- TRR~358. Maglione and
 Voll acknowledge the hospitality of the Bangalore International
 Center for Theoretical Sciences (ICTS) in December 2024 and the
 Mathematisches Forschungsinstitut Oberwolfach in June 2025.
\end{acknowledgements}

\bibliography{bibliography} \bibliographystyle{abbrv}

\end{document}